\documentclass[10pt]{article}
\usepackage{standalone,tikz}
\usepackage{graphics,verbatim}
\usepackage{amssymb,amsmath,amsthm}
\usepackage[font=footnotesize]{caption}

\usepackage{fullpage}
\newtheorem{theorem}{Theorem}[section]
\newtheorem{thm}[theorem]{Theorem}

\newtheorem{conjecture}[theorem]{Conjecture}

\newtheorem{lemma}[theorem]{Lemma}
\newtheorem{lem}[theorem]{Lemma}

\newtheorem{claim}{Claim}

\theoremstyle{definition}
\newtheorem{definition}[theorem]{Definition}
\newtheorem{defn}[theorem]{Definition}

\newcommand\Z{\mathbb{Z}}
\RequirePackage{marginnote,hyperref}
\newcommand{\aside}[1]{\marginnote{\scriptsize{#1}}[0cm]}
\newcommand{\aaside}[2]{\marginnote{\scriptsize{#1}}[#2]}
\newcommand\Emph[1]{\emph{#1}\aside{#1}}
\newcommand\EmphE[2]{\emph{#1}\aaside{#1}{#2}}

\newcommand\dit{^{\circ}}
\newcommand\diit{^{\circ\circ}}
\newcommand\diiit{^{\circ\circ\circ}}

\newcommand{\F}{\mathcal{F}}
\renewcommand{\P}{\mathcal{P}}
\newcommand{\Q}{\mathcal{Q}}
\newcommand{\T}{\mathcal{T}}

\def\aftermath{\par\vspace{-\belowdisplayskip}\vspace{-\parskip}\vspace{-\baselineskip}}

\newcommand{\JCTB}{{\it J. Combin. Theory Ser. B.}, }
\newcommand{\JGT}{{\it J. Graph Theory}, }
\newcommand{\ComHung}{{\it Combinatorica}, }
\newcommand{\DM}{{\it Discrete Math.}, }
\newcommand{\DAM}{{\it Discrete Appl. Math.}, }
\newcommand{\SIAMDM}{{\it SIAM J. Discrete Math.}, }
\newcommand{\JLMS}{{\it J. London Math. Soc.}, }

\begin{document}

\title{  Circular Flows in Planar Graphs}

\author{Daniel W. Cranston\thanks{Department of Mathematics and Applied Mathematics,
Virginia Commonwealth University, Richmond, VA, USA;
\texttt{dcranston@vcu.edu}; This research is partially supported by NSA Grant
H98230-15-1-0013.} \and Jiaao Li\thanks{School of Mathematical Sciences and LPMC
 Nankai University, Tianjin 300071, China;
\texttt{lijiaao@nankai.edu.cn}}}
\date{}
\maketitle

\begin{abstract}
For integers $a\ge 2b>0$, a \emph{circular $a/b$-flow} is a flow that takes values
from $\{\pm b, \pm(b+1), \dots, \pm(a-b)\}$.
The Planar Circular Flow
Conjecture states that every $2k$-edge-connected planar graph admits a circular
$(2+\frac{2}{k})$-flow.
The cases $k=1$ and $k=2$ are equivalent to
the Four Color Theorem and Gr\"{o}tzsch's 3-Color Theorem.  For $k\ge 3$,
the conjecture remains open.
Here we make progress when $k=4$ and $k=6$.  We prove that
(i) {\em every $10$-edge-connected planar graph admits a circular
$5/2$-flow} and (ii) {\em every $16$-edge-connected planar graph admits a
circular $7/3$-flow.}  The dual version of statement (i) on circular
coloring was previously proved  by Dvo\v{r}\'{a}k and Postle (Combinatorica
2017), but our proof has the advantages of being much shorter and avoiding the
use of computers for case-checking.  Further, it has new implications
for antisymmetric flows. Statement (ii) is especially interesting because
the counterexamples to Jaeger's original Circular Flow Conjecture are
$12$-edge-connected nonplanar graphs that admit no circular $7/3$-flow.
Thus, the planarity hypothesis of (ii) is essential.
\end{abstract}

\section{Introduction}
\subsection{Planar Circular Flow Conjecture}
For integers $a\ge 2b>0$, a \Emph{circular
$a/b$-flow}\footnote{Jaeger~\cite{Jaeger1988} showed that if $p,q,r,s\in \Z^+$
and $p/q=r/s$, then each graph $G$ has a circular $p/q$-flow if and only if it
has a circular $r/s$-flow.  (See~\cite{Goddyn1998} for more details.) We use
this result implicitly in the present paper.}
 is a flow that takes values from $\{\pm b,
\pm(b+1), \dots, \pm(a-b)\}$.  In this paper we study the following conjecture,
which arises from Jaeger's Circular Flow Conjecture \cite{Jaeger1988}.

\begin{conjecture}[Planar Circular Flow Conjecture]
\label{CONJ: PCFC}~\\
Every $2k$-edge-connected planar graph admits a circular
$(2+\frac{2}{k})$-flow.
\end{conjecture}

When $k=1$ this conjecture is the flow version of the 4 Color Theorem. It is
true for planar graphs (by 4CT), but false for nonplanar graphs because of the
Petersen graph, and all other snarks. Tutte's $4$-Flow
Conjecture, from 1966, 
claims that Conjecture~\ref{CONJ: PCFC}
extends to every graph with no Petersen minor.  When $k=2$, Conjecture~\ref{CONJ:
PCFC} is the dual of Gr\"{o}tzsch's 3-Color Theorem.  Tutte's $3$-Flow
Conjecture, from 1972, asserts that it extends to all graphs (both planar and
nonplanar).  In 1981 Jaeger further extended Tutte's Flow Conjectures, by
proposing a general Circular Flow Conjecture:
{\em for each even integer $k\ge 2$, every $2k$-edge-connected graph admits a
circular $(2+\frac{2}{k})$-flow}.  That is, he believed Conjecture~\ref{CONJ:
PCFC} extends to all graphs for all even $k$. A weaker version of Jaeger's
conjecture was proved by Thomassen~\cite{Thomassen2012}, for graphs with edge
connectivity at least $2k^2+k$.  This edge connectivity condition was
substantially improved by Lov\'asz, Thomassen, Wu, Zhang~\cite{LTWZ2013}.

\begin{theorem}{\em (Lov\'asz, Thomassen, Wu, Zhang~\cite{LTWZ2013})}
For each even integer $k\ge 2$, every $3k$-edge-connected graph admits a circular $(2+\frac{2}{k})$-flow.
\label{LTWZ-thm}
\end{theorem}

In contrast, Jaeger's Circular Flow Conjecture was recently disproved for all
$k\ge 6$. 
In~\cite{HLWZ}, for each even integer $k\ge 6$,
the authors construct a $2k$-edge-connected nonplanar graph admitting no circular $(2+\frac{2}{k})$-flow.
And for large odd integers $k$, we can also modify the construction
in~\cite{HLWZ} to get $2k$-edge-connected nonplanar graphs admitting no
circular $(2+\frac{2}{k})$-flow.  Thus, the planarity hypothesis of
Conjecture~\ref{CONJ: PCFC} seems essential.
The case $k=4$ of Jaeger's Circular Flow Conjecture, which remains open, is
particularly important, since Jaeger \cite{Jaeger1988} observed that if every
$9$-edge-connected graph admits a circular $5/2$-flow, then
Tutte's celebrated $5$-Flow Conjecture follows.

Our main theorems improve on Theorem~\ref{LTWZ-thm}, restricted to planar graphs, when
$k\in\{4,6\}$. 
\begin{theorem}
  Every $10$-edge-connected planar graph admits a circular $5/2$-flow.
\label{5/2-flow-thm}
\end{theorem}

\begin{theorem}
  Every $16$-edge-connected planar graph admits a circular $7/3$-flow.
\label{7/3-flow-thm}
\end{theorem}

The dual version of Theorem~\ref{5/2-flow-thm}, on circular
coloring, was proved by Dvo\v{r}\'{a}k and Postle~\cite{DP2017}. In
fact, their coloring result holds for a larger class of graphs that includes
some sparse nonplanar graphs, as well as all planar graphs with girth at
least 10.
However, our proof is much shorter and avoids using computers for
case-checking.  Our proof also has new implications for antisymmetric
flows (see Theorem~\ref{antisymmetric-thm} below). 
Theorem~\ref{7/3-flow-thm} is especially interesting because
the counterexamples in~\cite{HLWZ} to Jaeger's original circular flow
conjecture are $12$-edge-connected nonplanar graphs that admit no circular
$7/3$-flow.

\subsection{Circular Flows and Modulo Orientations}
Graphs in this paper are finite and can have multiple edges, but no loops. Our notation is mainly
standard.  For a graph $G$, we write $|G|$ for
$|V(G)|$ and write $\|G\|$ for $|E(G)|$\aside{$|G|,\|G\|$}.
Let $\delta(G)$ denote the minimum degree in a graph $G$.  A $k$-vertex is a
vertex of degree $k$.
For disjoint vertex subsets $X$ and $Y$, let \Emph{$[X,Y]_G$} denote the set of edges
in $G$ with one endpoint in each of $X$ and $Y$.  Let $X^c=V(G)\setminus
X$\aaside{$X^c$, $d(X)$}{0mm}, and
let $d(X)=|[X,X^c]|$.  For vertices $v$ and $w$, let
$\mu(vw)=|[\{v\},\{w\}]_G|$ and $\mu(G)=\max_{v,w\in
V(G)}\mu(vw)$\aaside{$\mu(vw),\mu(G)$}{-4mm}.

To \Emph{lift} a pair of edges $w_1v$, $vw_2$ incident to a vertex $v$ in a
graph $G$ means to delete $w_1v$ and $vw_2$ and create a new edge $w_1w_2$.
 To \Emph{contract} an edge $e$ in $G$ means to identify its two endpoints and
then delete the resulting loop.  For a subgraph $H$ of $G$, we write $G/H$ to
denote the graph formed from $G$ by successively contracting the edges of
$E(H)$.  The lifting and contraction operations are used frequently in this
paper.

An orientation $D$ of a graph $G$ is a \EmphE{modulo $(2p+1)$-orientation}{-2mm}
if $d^+_D(v)-d^-_D(v)\equiv 0\pmod{2p+1}$ for each $v\in V(G)$.    By the
following lemma of Jaeger~\cite{Jaeger1988}, this problem is equivalent to
finding circular flows (for a short proof, see~\cite[Theorem 9.2.3]{Zhang-book}).
\begin{lem}\cite{Jaeger1988}
A graph admits a circular $(2+\frac{1}{p})$-flow if and only if it has a modulo $(2p+1)$-orientation.
\label{prop1}
\end{lem}
To prove our results, we study modulo orientations.
Let $G$ be a graph. A function $\beta: V(G) \mapsto \Z_{2p+1}$ is a 
\EmphE{$\Z_{2p+1}$-boundary}{-4mm} if $\sum_{v\in V(G)}\beta(v)\equiv
0\pmod{2p+1}$. Given a
$\Z_{2p+1}$-boundary $\beta$, a \EmphE{$(\Z_{2p+1},\beta)$-orientation}{2mm} is
an orientation $D$ such that $d_D^+(v)-d_D^-(v)\equiv \beta(v) \pmod{2p+1}$ for
each $v\in V(G)$. When such an orientation exists, we say that the boundary
$\beta$ is \EmphE{achievable}{-1mm}.  If $\beta(v)=0$ for all $v\in V(G)$, then a
$(\Z_{2p+1},\beta)$-orientation is simply a modulo $(2p+1)$-orientation.
As defined in \cite{Lai2007, Lai2014}, a graph $G$ is
\EmphE{strongly $\Z_{2p+1}$-connected}{-2mm} if for any $\Z_{2p+1}$-boundary
$\beta$, graph $G$ admits a $(\Z_{2p+1},\beta)$-orientation.
When the context is clear,  we may
simply write \EmphE{$\beta$-orientation}{4mm} for
$(\Z_{2p+1},\beta)$-orientation.  Suppose we are given a graph $G$, an
integer $p$, a $\Z_{2p+1}$-boundary $\beta$ for $G$, and a connected
subgraph $H\subsetneq G$.  We form $G'$ from $G$ by contracting $H$; that is $G'=G/H$.
 Let $w$ denote the new vertex in $G'$, formed by
contracting $E(H)$.  Define $\beta'$ for $G'$ by $\beta'(v)=\beta(v)$ for each
$v\in V(G')\setminus\{w\}$, and $\beta'(w)=\sum_{v\in V(H)}\beta(v) \pmod{2p+1}$.
Note that $\beta'$ is a $\Z_{2p+1}$-boundary for $G'$.  The motivation for
generalizing modulo orientations is the following observation of
Lai~\cite{Lai2007}, which is also applied in Thomassen et al.~\cite{Thomassen2012,LTWZ2013}.
\begin{lem}[\cite{Lai2007}]
\label{reduc-lem}
Let $G$ be a graph with a subgraph $H$, and let $G'=G/H$.  Let $\beta$ and
$\beta'$ be $\Z_{2p+1}$ boundaries (respectively) of $G$ and $G'$, as defined
above. If $H$ is strongly $\Z_{2p+1}$-connected, then
every $\beta'$-orientation of $G'$ can be extended to a  $\beta$-orientation of
$G$.  In particular, each of the following holds.
\begin{enumerate}
\item[(i)] If $H$ is strongly $\Z_{2p+1}$-connected and $G/H$ has a
modulo $(2p+1)$-orientation, then $G$ has a modulo $(2p+1)$-orientation.
\item[(ii)] If $H$ and $G/H$ are strongly $\Z_{2p+1}$-connected, then $G$ is
also strongly $\Z_{2p+1}$-connected.
\end{enumerate}
\end{lem}

\begin{proof}
We prove the first statement, since it implies (i) and (ii).  Fix a
$\beta'$-orientation of $G'$.  This yields an orientation $D$ of the subgraph
$G-E(G[V(H)])$. By orienting arbitrarily each edge in $E(G[V(H)])\setminus
E(H)$, we obtain a $\beta''$-orientation $D_1$ of $G-E(H)$, for some $\beta''$.
 For each $v\in V(H)$, let $\gamma(v)=\beta(v)-\beta''(v)$.  It is easy to
check that $\gamma$ is a $\Z_{2p+1}$-boundary of $H$.  Since $H$ is strongly
$\Z_{2p+1}$-connected, $H$ has a $\gamma$-orientation $D_2$.  Hence $D_1\cup
D_2$ is a $\beta$-orientation of $G$.
\end{proof}

~

\noindent
{\bf Proof Outline for Main Results.} To prove Theorems~\ref{5/2-flow-thm} and \ref{7/3-flow-thm},
we actually establish two stronger, more technical results on orientations;
namely, we prove Theorems~\ref{THM: Main1} and \ref{THM: Main2}.
Lemma~\ref{reduc-lem} shows that strongly $\Z_{2p+1}$-connected graphs are
contractible configurations when we are looking for modulo orientations. To
prove Theorems~\ref{THM: Main1} and \ref{THM: Main2},
we use lifting and contraction operations to find many more reducible
configurations. These configurations eventually facilitate a discharging proof.
The proofs of Theorems~\ref{5/2-flow-thm} and~\ref{7/3-flow-thm} are similar,
though the latter is harder.  In the next section we just discuss
Theorem~\ref{5/2-flow-thm}, but most of the key ideas are reused in the proof of
Theorem~\ref{7/3-flow-thm}.

\section{Circular $5/2$-flows: Proof of Theorem~\ref{5/2-flow-thm}}
\label{Z5-sec}
\subsection{Modulo $5$-Orientations and Antisymmetric $\Z_5$-flows}

To prove Theorem~\ref{5/2-flow-thm},
we will first present a more technical result, Theorem~\ref{THM: Main1},
which yields Theorem~\ref{5/2-flow-thm} as an easy corollary
(as we show below in Theorem~\ref{10-edge-thm}).
The hypothesis in Theorem~\ref{THM: Main1} uses a weight function $w$, which is
motivated by the following Spanning Tree Packing Theorem of
Nash-Williams \cite{Nash1961} and Tutte \cite{Tutte1961}: 
{\em a graph $G$ has $k$ edge-disjoint spanning trees if and
only if every partition ${\mathcal P}=\{P_1, P_2,\dots, P_t\}$ satisfies
$\sum_{i=1}^{t}d(P_i)-2k(t-1)\ge 0$.}  This condition is necessary,
since in a partition with $t$ parts, each spanning tree has at least $t-1$ edges
between parts.  It is shown in~\cite[Proposition~3.9]{LaLL17}  that if $G$
is strongly $\Z_{2p+1}$-connected, then it contains $2p$ edge-disjoint spanning
trees (although this necessary condition is not always sufficient).  To capture
this idea, we define the following weight function.

\begin{definition}
\label{DEF: partition}
Let ${\mathcal P}=\{P_1, P_2,\dots, P_t\}$ be a partition of $V(G)$. Let
$$w_G({\mathcal P})=\sum_{i=1}^{t}d(P_i)-11t+19$$ and
$$w(G)=\min\{w_G({\mathcal P}): {\mathcal P}\ is\ a\ partition\ of\ V(G)\}.$$
\end{definition}

Let \Emph{$T_{a,b,c}$} denote a 3-vertex graph (triangle) with its pairs of vertices
joined by $a$, $b$, and $c$ parallel edges; let \Emph{$aH$} denote the graph
formed from $H$ by replacing each edge with $a$ parallel
edges. 
For example,  $w(3K_2)=3$, $w(2K_2)=1$, $w(T_{2,2,3})=w(T_{1,3,3})=0$; see
Figure~\ref{FIG: K23J12}.  For each of these four graphs the minimum in the
definition of $w(G)$ is attained only by the partition with each vertex in its
own part.
We typically assume $V(T_{a,b,c})=\{v_1,v_2,v_3\}$ and $d(v_1)\le d(v_2)\le
d(v_3)$.

\begin{figure}[t]

\setlength{\unitlength}{0.08cm}

\begin{center}

\begin{picture}(150,30)
\put(0,10){\circle*{2}}\put(20,10){\circle*{2}}
\qbezier(0, 10)(0, 10)(20, 10)\qbezier(0, 10)(10, 15)(20, 10)\qbezier(0, 10)(10, 5)(20, 10)
\put(5,-7){\footnotesize{$3K_2$}}

\put(40,10){\circle*{2}}\put(60,10){\circle*{2}}
\qbezier(40, 10)(50, 15)(60, 10)\qbezier(40, 10)(50, 5)(60, 10)
\put(46,-7){\footnotesize{$2K_2$}}

\put(80,5){\circle*{2}}\put(100,5){\circle*{2}}\put(90,25){\circle*{2}}
\qbezier(80, 5)(90, 10)(100, 5)\qbezier(80, 5)(90, 5)(100, 5)\qbezier(80, 5)(90, 0)(100, 5)
\qbezier(80, 5)(85, 11)(90, 25)\qbezier(80, 5)(84, 17)(90, 25)
\qbezier(100, 5)(95, 10)(90, 25)\qbezier(100, 5)(94, 21)(90, 25)
\put(86,-7){\footnotesize{$T_{2,2,3}$}}

\put(120,5){\circle*{2}}\put(140,5){\circle*{2}}\put(130,25){\circle*{2}}
\qbezier(120, 5)(130, 10)(140, 5)\qbezier(120, 5)(130, 5)(140, 5)\qbezier(120, 5)(130, 0)(140, 5)
\qbezier(120, 5)(125, 15)(130, 25)
\qbezier(140, 5)(132, 10)(130, 25)\qbezier(140, 5)(137, 22)(130, 25)\qbezier(140, 5)(135, 15)(130, 25)
\put(127,-7){\footnotesize{$T_{1,3,3}$}}

\end{picture}
\end{center}
\vspace{0.4cm}
\caption{The graphs $3K_2, 2K_2, T_{2,2,3}, T_{1,3,3}$.}
\label{FIG: K23J12}
\end{figure}
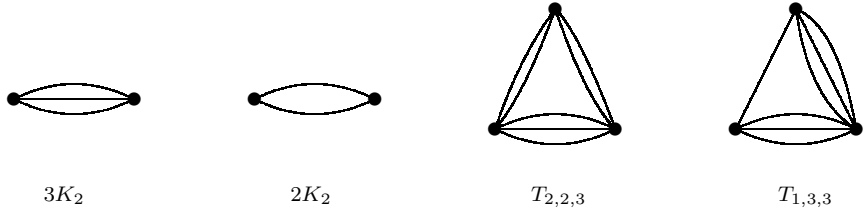

Let $\T=\{2K_2, 3K_2, T_{2,2,3}, T_{1,3,3}\}$. Each graph $G\in \T$ (see Figure~\ref{FIG: K23J12}) is not strongly $\Z_5$-connected,
since there exists some $\Z_5$-boundary $\beta$ for which $G$ has no
$\beta$-orientation.  A short case analysis shows that none of the
following boundaries are achievable.
For $3K_2$, let $\beta(v_1)=\beta(v_2)=0$.  For
$2K_2$, let $\beta(v_1)=1$ and $\beta(v_2)=4$.
For $T_{2,2,3}$, let $\beta(v_1)=1$ and $\beta(v_2)=\beta(v_3)=2$.
For $T_{1,3,3}$, let $\beta(v_1)=\beta(v_2)=1$ and $\beta(v_3)=3$.

Now suppose that $G$ has a partition $\P$ such that $G/\P\in\T$,
where the vertices in each $P_i$ are identified to form $v_i$.
To construct a $\Z_5$-boundary $\gamma$ for
which $G$ has no $\gamma$-orientation, we assign boundary  $\gamma$
so that $\sum_{v\in P_i}\gamma(v)\equiv \beta(v_i)$.  Hence $G$ has
no $\gamma$-orientation precisely because $G/\P$ has no
$\beta$-orientation.
%
%
We call a partition $\P$ \EmphE{troublesome}{-4mm} if  $G/\P\in \T=\{2K_2, 3K_2, T_{2,2,3}, T_{1,3,3}\}$.  The main result of Section~\ref{Z5-sec} is Theorem~\ref{THM: Main1}.
\begin{theorem}
\label{THM: Main1}
Let $G$ be a planar graph and $\beta$ be a $\Z_5$-boundary of $G$. If $w(G)\ge
0$, then $G$ admits a $(\Z_5,\beta)$-orientation, unless $G$ has a troublesome partition.
\end{theorem}

\smallskip

Before proving Theorem~\ref{5/2-flow-thm}, we prove a slightly
weaker result, assuming the truth of Theorem~\ref{THM: Main1}.

\begin{thm}\label{11-edge-thm}
If $G$ is an 11-edge-connected planar graph, then $G$ is strongly
$\Z_5$-connected.
\end{thm}
\begin{proof} 
Let $G$ be an 11-edge-connected planar graph.  Fix a partition $\P$.  Since $G$
is 11-edge-connected, $d(P_i)\ge 11$ for each $i$, which implies $w_G(\P)\ge 19$.
Thus $w(G)\ge 19$.  Since it is easy to see  each troublesome partition $\P$ has $w(G/\P)\le
3$,  we obtain that $G$ has no partition $\P$ such that $G/\P$ is troublesome.  Now
Theorem~\ref{THM: Main1} implies that $G$ is strongly $\Z_5$-connected.
\end{proof}

An \Emph{antisymmetric $\Z_5$-flow} in a directed graph
$D=D(G)$ is a $\Z_5$-flow  such that no two edges have flow values
summing to 0.  One example is any $\Z_5$-flow that uses only values 1 and 2.
Esperet, de Verclos, Le, and Thomass\'{e} \cite{EJLT2017} proved that if a
graph $G$ is strongly $\Z_5$-connected, then every orientation $D(G)$ of $G$
admits an antisymmetric $\Z_5$-flow.  Together with work of Lov\'asz et
al.~\cite{LTWZ2013}, this implies that every directed
$12$-edge-connected graph admits an antisymmetric $\Z_5$-flow.
Esperet et~al.~\cite{EJLT2017} conjectured the stronger result that
{\em every directed $8$-edge-connected  graph  admits an antisymmetric
$\Z_5$-flow}.
The concept of antisymmetric flows and its dual, homomorphisms to oriented
graphs, were introduced by Ne\v{s}et\v{r}il
and Raspaud \cite{NR1999}. In \cite{NRS1997}, Ne\v{s}et\v{r}il, Raspaud and
Sopena showed that every orientation of a planar graph of girth
at least $16$ has a homomorphism to an oriented simple graph on at most 5
vertices. The girth condition is reduced to $14$ in \cite{BKNRS1999}, to $13$
in \cite{BKKW2004}, and finally to $12$ in \cite{BIK2007}.  By duality,
the results of
\cite{NR1999}, 
\cite{EJLT2017}, and 
\cite{LTWZ2013} combine to imply that girth $12$ suffices.
After the girth $12$ result of  Borodin et al.~\cite{BIK2007} in 2007,
Esperet et al.~\cite{EJLT2017} remarked that ``it is not known whether the same
holds for planar graphs of girth at least $11$.''
Note that the result of Dvo\v{r}\'{a}k and Postle~\cite{DP2017} does not seem
 to apply to homomorphisms to oriented graphs.  By Theorem~\ref{11-edge-thm}, we improve
this girth bound for planar graphs.

\begin{thm}
\label{antisymmetric-thm}
Every directed $11$-edge-connected planar graph admits an antisymmetric
$\Z_5$-flow. Dually, every orientation of a planar graph of girth at least $11$
has a homomorphism to an oriented simple graph on at most 5 vertices.
\end{thm}

A graph $G$ has \Emph{odd edge-connectivity} $t$ if the smallest edge cut of
odd size has size $t$.  Our strongest result on modulo $5$-orientations is the
following, which includes Theorem~\ref{5/2-flow-thm} as a special case.

\begin{theorem}
Every odd-$11$-edge-connected planar graph admits a modulo 5-orientation.
In particular, every 10-edge-connected planar graph admits a modulo
5-orientation (and thus a circular 5/2-flow).
\label{10-edge-thm}
\end{theorem}
\begin{proof}
The second statement follows from the first, by Lemma~\ref{prop1}.
To prove the first, suppose the theorem is false, and let $G$ be a
counterexample minimizing $\|G\|$.  By Zhang's Splitting Lemma\footnote{This says
that if $G$ has a vertex $v$ with $d(v)\notin\{2,11\}$, then
we can lift a pair of edges incident to $v$ that are successive in the circular
order around $v$, and the resulting graph is still planar and
odd-11-edge-connected.  For example, if $d(v)=10$, then all edges
incident to $v$ will be lifted in pairs, so the boundary value at $v$ in the
resulting orientation will be 0.  This is why the proof yields a modulo
5-orientation, but does not show that $G$ is strongly $\Z_5$-connected.}
for odd edge-connectivity~\cite{CQsplitting02}, we know $\delta(G)\ge
11$.  If $G$ is $11$-edge-connected,
then we are done by Theorem~\ref{11-edge-thm}; so assume it is not. Choose a
smallest set $W\subset V(G)$ such that $d(W)<11$.  Note that $|W|\ge 2$, and
every proper subset $W'\subsetneq W$ satisfies $d(W')\ge 11$. Let $H=G[W]$.
For any partition $\P=\{P_1, P_2,\dots, P_t\}$ of $H$ with $t\ge 2$, we know
that $d_G(P_i)\ge 11$ by the minimality of $W$, since $P_i\subsetneq W$.  This implies
  \begin{eqnarray*}
    w_H({\mathcal P})&=&\sum_{i=1}^td_H(P_i)-11t+19\\
    &=&\sum_{i=1}^td_G(P_i)-d_G(W^c)-11t+19\\
    &>& 11t-11-11t+19\ge 8.
  \end{eqnarray*}
Thus $w(H)\ge 9$, which implies $H$ is strongly $\Z_5$-connected by Theorem~\ref{THM:
Main1}. By the minimality of $G$, the graph $G/H$ has a modulo 5-orientation.
By Lemma~\ref{reduc-lem}, this extends to a modulo 5-orientation of $G$, which
completes the proof.
\end{proof}

\subsection{Reducible Configurations and Partitions}
\label{sec:Z5-prelims}

To prove Theorem~\ref{THM: Main1}, we assume the result is false and study a minimal
counterexample.  In the next subsection we prove many
structural results about the minimal counterexample, which ultimately imply it cannot exist.
In this subsection we prove that a few small graphs cannot appear as subgraphs of
the minimal counterexample.  We call such a forbidden subgraph
\Emph{reducible}.  By Lemma~\ref{reduc-lem}, to show that $H$ is reducible it
suffices to show $H$ is strongly $\Z_5$-connected.

%

Let $G$ be a graph. We often lift a pair of edges $w_1v$, $vw_2$ incident to a
vertex $v$ in $G$ to form a new graph $G'$. That is, we delete $w_1v$
and $vw_2$ and create a new edge $w_1w_2$.  If $G'$ is strongly
$\Z_k$-connected, then so is $G$, since from any $\beta$-orientation of $G'$ we
delete the edge $w_1w_2$   and add the directed edges $w_1v$ and $vw_2$ to
obtain a $\beta$-orientation of $G$.  To prove $G$ is strongly
$\Z_k$-connected, we use lifting in two similar ways.

First, we lift some edge pairs to create a $G'$ that contains a strongly
$\Z_k$-connected subgraph $H$.  If $G'/H$ is strongly $\Z_k$-connected, then so
is $G'$ by Lemma~\ref{reduc-lem}.  As discussed in the previous paragraph, so is $G$.
Second, given a $\Z_k$-boundary $\beta$, we orient some edges incident to a
vertex $v$ to achieve $\beta(v)$.  For each edge $vw$ that we orient, we
increase or decrease by 1 the value of $\beta(w)$.  Now we delete $v$ and all
oriented edges, and lift the remaining edges incident to $v$ (in pairs).  Call
the resulting graph and boundary $G'$ and $\beta'$.  If $G'$ has a
$\beta'$-orientation, then $G$ has a $\beta$-orientation.  We call
these \EmphE{lifting reductions of the first and second type}{-10mm}, respectively.  In this
paper whenever we lift an edge pair $vw$, $wx$ we require that edge $vx$ already
exists. Thus,
our lifting reductions always preserve planarity.


\begin{lemma}
\label{LEM: 4K2J2K4inSZ5}
Each of the graphs $4K_2, T_{2,3,3}$, $2K_4$, and $3C_4$, shown in
Figure~\ref{FIG: 4K2J2K4C4}, is strongly $\Z_5$-connected.
\end{lemma}
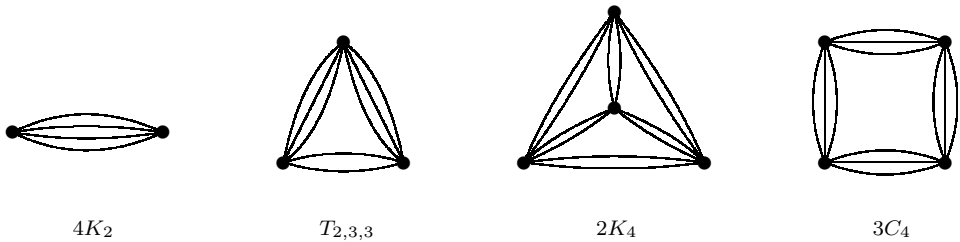
\begin{figure}[ht]

\setlength{\unitlength}{0.08cm}

\begin{center}

\begin{picture}(165,30)
\put(0,10){\circle*{2}}\put(25,10){\circle*{2}}
\qbezier(0, 10)(12, 12)(25, 10)\qbezier(0, 10)(12, 16)(25, 10)\qbezier(0, 10)(12, 8)(25, 10)\qbezier(0, 10)(12, 4)(25, 10)
\put(10,-7){\footnotesize{$4K_2$}}

\put(45,5){\circle*{2}}\put(65,5){\circle*{2}}\put(55,25){\circle*{2}}
\qbezier(45, 5)(55, 8)(65, 5)\qbezier(45, 5)(55, 2)(65, 5)
\qbezier(45, 5)(52, 10)(55, 25)\qbezier(45, 5)(48, 20)(55, 25)\qbezier(45, 5)(45, 5)(55, 25)
\qbezier(65, 5)(65, 5)(55, 25)\qbezier(65, 5)(58, 10)(55, 25)\qbezier(65, 5)(62, 20)(55, 25)
\put(51,-7){\footnotesize{$T_{2,3,3}$}}

\put(85,5){\circle*{2}}\put(115,5){\circle*{2}}\put(100,30){\circle*{2}}\put(100,14){\circle*{2}}
\qbezier(85, 5)(100, 7)(115, 5)\qbezier(85, 5)(100, 3)(115, 5)
\qbezier(85, 5)(90, 17)(100, 30)\qbezier(85, 5)(94, 17)(100, 30)
\qbezier(115, 5)(110, 17)(100, 30)\qbezier(115, 5)(106, 17)(100, 30)
\qbezier(85, 5)(90, 10)(100, 14)\qbezier(85, 5)(94, 9)(100, 14)
\qbezier(115, 5)(110, 10)(100, 14)\qbezier(115, 5)(106, 9)(100, 14)
\qbezier(100, 30)(98, 22)(100, 14)\qbezier(100,30)(102, 22)(100, 14)

\put(97,-7){\footnotesize{$2K_4$}}

\put(143,-7){\footnotesize{$3C_4$}}
\put(135,5){\circle*{2}}\put(155,5){\circle*{2}}\put(135,25){\circle*{2}}\put(155,25){\circle*{2}}
\qbezier(135,5)(135,5)(155,5)\qbezier(135,5)(145,1)(155,5)\qbezier(135,5)(145,9)(155,5)
\qbezier(135,5)(135,15)(135,25)\qbezier(135,5)(131,15)(135,25)\qbezier(135,5)(139,15)(135,25)
\qbezier(135,25)(135,25)(155,25)\qbezier(135,25)(145,21)(155,25)\qbezier(135,25)(145,29)(155,25)
\qbezier(155,5)(155,15)(155,25)\qbezier(155,5)(151,15)(155,25)\qbezier(155,5)(159,15)(155,25)

\end{picture}
\end{center}
\vspace{0.4cm}
\caption{The graphs $4K_2, T_{2,3,3}, 2K_4, 3C_4$.}
\label{FIG: 4K2J2K4C4}
\end{figure}
\begin{proof}
Proving the lemma amounts to checking a finite list of cases.  So our goal
is to make this as painless as possible.
%
Throughout we fix a $\Z_5$-boundary
$\beta$ and construct an orientation that achieves $\beta$.

Let $G=4K_2$ and $V(G)=\{v_1,v_2\}$.  To achieve $\beta(v_1)\in\{0,1,2,3,4\}$
the number of edges we orient out of $v_1$ is (respectively) 2, 0, 3, 1, 4.

Let $G=T_{2,3,3}$ and $V(G)=\{v_1,v_2,v_3\}$, with $d(v_1)=d(v_2)=5$ and
$d(v_3)=6$.  If $\beta(v_1)\ne 0$, then we achieve $\beta$ by orienting 3 edges
incident to $v_1$, and lifting a pair of unused, nonparallel, edges incident to
$v_1$ to create a fourth edge $v_2v_3$.  Since $4K_2$ is strongly
$\Z_5$-connected, we can use the resulting 4 edges to achieve $\beta(v_2)$ and
$\beta(v_3)$.  (This is a lifting reduction of the second type.  In what follows,
we are less explicit about such descriptions.)  So we assume that
$\beta(v_1)=0$ and, by symmetry, $\beta(v_2)=0$.  This implies $\beta(v_3)=0$.
Now we orient all edges from $v_1$ to $v_3$, from $v_1$ to $v_2$ and from $v_3$ to $v_2$.

Let $G=2K_4$ and $V(G)=\{v_1,v_2,v_3,v_4\}$.  If $\beta(v_1)\in\{0,2,3\}$, then
we achieve $\beta(v_1)$ by orienting two nonparallel edges incident to $v_1$.
Now we lift two pairs of unused edges incident to $v_1$ to get a $T_{2,3,3}$.
Since $T_{2,3,3}$ is strongly $\Z_5$-connected, we are done by
Lemma~\ref{reduc-lem}.  So assume $\beta(v_1)\notin\{0,2,3\}$.  By symmetry, we
assume $\beta(v_i)\in\{1,4\}$ for all $i$.  Since $\beta$ is a $\Z_5$-boundary,
we further assume $\beta(v_i)=1$ when $i\in\{1,2\}$ and $\beta(v_j)=4$ when
$j\in\{3,4\}$.  Let $V_1=\{v_1,v_2\}$ and $V_2=\{v_3,v_4\}$.  Orient all edges
from $V_2$ to $V_1$.  For each pair of parallel edges within $V_1$ or $V_2$,
orient one edge in each direction.  This achieves $\beta$.

Let $G=3C_4$ and $V(G)=\{v_1,v_2,v_3,v_4\}$ with $v_1,v_3\in N(v_2)\cap N(v_4)$.
If $\beta(v_1)\in\{0,2,3\}$, then we achieve $\beta(v_1)$ by orienting two
nonparallel edges incident to $v_1$ and lifting two pairs of edges incident
to $v_1$.  The resulting unoriented graph is $T_{2,3,3}$, so we are done by
Lemma~\ref{reduc-lem}.  Assume instead, by symmetry, that
$\beta(v_i)\in\{1,4\}$ for all $i$.
Since $\beta$ is a $\Z_5$-boundary, two vertices $v_i$ have $\beta(v_i)=1$ and
two vertices $v_j$ have $\beta(v_j)=4$.  By symmetry, assume $\beta(v_1)=1$.
If $\beta(v_3)=1$, then orient all edges out from $v_1$ and $v_3$.  Assume
instead, by symmetry, that $\beta(v_2)=1$; now reverse one edge $v_3v_2$ from
the previous orientation.
\end{proof}

\begin{defn}
For partitions ${\P}=\{P_1, P_2,\dots, P_t\}$ and ${\P}'=\{P_1', P_2',\dots,
P_s'\}$, we say that $\P'$ is a \EmphE{refinement}{0mm} of $\P$, denoted by
${\P}'\preceq{\P}$, if ${\P}'$ is obtained from ${\P}$ by further partitioning
$P_i$ into smaller sets for some $P_i$'s in ${\P}$.  More formally, we require
that for every $P_j'\in {\P}'$, there exists $P_i\in {\P}$ such that $P_j'\subseteq
P_i$.  

Since partitions are central to our theorems and proofs, we name a few common
types of them. A partition ${\P}=\{P_1, P_2,\dots, P_t\}$ is \EmphE{trivial}{-3mm} if each
part $P_i$ is a singleton, i.e., $V(G)$ is partitioned into $|G|$
parts; otherwise $\P$ is \EmphE{nontrivial}{-3mm}.  A trivial partition is minimal under the
relation $\prec$.  A partition ${\P}=\{P_1, P_2,\dots, P_t\}$ is \EmphE{almost
trivial}{-3mm} if $t=|G|-1$ and there is a unique part $P_i$ with $|P_i|=2$.
 A partition ${\P}$ is called \EmphE{normal}{4mm} if it is neither trivial
nor almost trivial and ${\P}\neq \{V(G)\}$.
\end{defn}

Given a partition $\P$ of $V(G)$ and a partition $\Q$ of $G[P_1]$, the
following lemma relates the weights of $\P$, $\Q$, and the refinement
$\Q\cup(\P\setminus\{P_1\})$.

\begin{lemma}
  Let ${\P}=\{P_1, P_2,\dots, P_t\}$ be a partition of $V(G)$ with $|P_1|>1$.
Let $H=G[P_1]$ and let ${\Q}=\{Q_1, Q_2,\dots, Q_s\}$ be a partition of
$V(H)$.  Now ${\Q}\cup ({\P}\setminus\{P_1\})$ is a refinement of $\P$
satisfying
\begin{eqnarray}
\label{EQ: wH}
  w_G({\Q}\cup ({\P}\setminus\{P_1\}))= w_H({\Q})+ w_G({\P})-8.
\end{eqnarray}
\label{mod5-key-lem}
\end{lemma}

\begin{proof} Clearly, ${\Q}\cup ({\P}\setminus\{P_1\})$ is a refinement
of $\P$, and it follows from Definition \ref{DEF: partition} that
\begin{align*}
w_G({\Q}\cup ({\P}\setminus\{P_1\}))
&=\sum_{i=1}^sd_G(Q_i) + \sum_{j=2}^td_G(P_j) -11(s+t-1)+19\\
&=[\sum_{i=1}^sd_G(Q_i)-d_G(P_1)-11s+19] + [\sum_{j=1}^td_G(P_j) -11(t-1)]\\
&=[\sum_{i=1}^sd_H(Q_i)-11s+19] + [\sum_{j=1}^td_G(P_j) -11t+19]-(19-11)\\
&= w_H({\Q})+ w_G({\P})-(19-11).
\end{align*}
\aftermath
\end{proof}

\subsection{Properties of a Minimal Counterexample to Theorem~\ref{THM: Main1}}

{\bf Let $G$ be a counterexample to Theorem \ref{THM: Main1}  that minimizes
$|G|+\|G\|$.} Thus Theorem~\ref{THM: Main1} holds for all graphs smaller
than $G$. This implies the following lemma, which we will use frequently.

\begin{lemma}\label{LEM: smallH}
If $H$ is a planar graph with $w(H)\ge 0$ and $|H|+\|H\|<|G|+\|G\|$, then each of the following holds.
\begin{itemize}
\item[(a)] If $w_H({\P})\ge 4$ for every nontrivial partition ${\P}$, then $H$ is
strongly $\Z_5$-connected unless $H\in\{2K_2, 3K_2$, $T_{1,3,3}, T_{2,2,3}\}$.
\item[(b)] If $w(H)\ge 1$ and $H$ is $4$-edge-connected, then $H$ is strongly $\Z_5$-connected.
\item[(c)] If $w(H)\ge 4$, then $H$ is strongly $\Z_5$-connected.
\end{itemize}
\end{lemma}
\begin{proof}
To prove each part, we fix a $\Z_5$-boundary $\beta$ and apply Theorem \ref{THM:
Main1} to $H$.  Notice that each troublesome partition $\P$ satisfies $w(G/\P)\le 3$.
So for (a), only the trivial partition can be troublesome.
Thus, $H$ is strongly $\Z_5$-connected unless $H\in\{2K_2, 3K_2, T_{1,3,3},
T_{2,2,3}\}$. For (b), $G$ has no partition $\P$ with $G/\P\in\{2K_2,3K_2\}$
since $G$ is $4$-edge-connected.
And $G$ has no partition $\P$ with $G/\P\in\{T_{1,3,3},T_{2,2,3}\}$
since $w(H)\ge 1$.  So $H$ is again strongly $\Z_5$-connected, by
Theorem~\ref{THM: Main1}.  Finally,
(c) follows from (b), since if $H$ has an edge cut $[X,X^c]$ of size at most 3,
then $w_H(\{X,X^c\})\le 2(3)-11(2)+19=3$, which contradicts our assumption that
$w(H)\ge 4$.  
\end{proof}

The main idea of our proof is to show that the value of the weight function
$w_G(\P)$ is relatively large for each nontrivial partition $\P$.  This
enables us to slightly modify certain proper subgraphs and still apply
Lemma~\ref{LEM: smallH} to the resulting graph $H$.  This added flexibility
(to slightly modify the subgraph) helps us to prove that more subgraphs are
reducible.  In the next section, these forbidden subgraphs facilitate a
discharging proof that shows that our minimal counterexample $G$ cannot exist.

\begin{claim}
\label{CL: nostrongZ5}
$G$ has no strongly $\Z_5$-connected subgraph $H$ with $|H|>1$. In particular,
\begin{itemize}
\item[(a)] $G$ has no copy of $4K_2$, $T_{2,3,3}$, $2K_4$, or $3C_4$ (by
Lemma~\ref{LEM: 4K2J2K4inSZ5}), and
\item[(b)] $|G|\ge 4$.
\end{itemize}
\end{claim}
\begin{proof}
Suppose to the contrary that $H$ is a strongly $\Z_5$-connected subgraph of $G$
with $|H|>1$, and let $G'=G/H$.  Since $G$ is a minimal counterexample, $G'$
is strongly $\Z_5$-connected, by Theorem~\ref{THM: Main1}.  So
Lemma~\ref{reduc-lem} implies $G$ is strongly $\Z_5$-connected, which is a
contradiction.  This proves both the first statement and (a).  For (b), clearly
$|G|\ge 3$, since $w(G)\ge 0$ and $G\notin\{2K_2,3K_2\}$ and $G$ contains no
$4K_2$.  So assume $|G|=3$.  Since $w(G/\P)\ge 0$ for the trivial partition $\P$,
we know that $\|G\|\ge 8$.  Since $G\notin\{T_{1,3,3},T_{2,2,3}\}$, either $G$
contains $4K_2$ or $G$ contains $T_{2,3,3}$.  Each case contradicts (a).
\end{proof}

\begin{claim}
\label{CL: nontrivialge4}
Let ${\P}=\{P_1, P_2,\dots, P_t\}$ be a nontrivial partition of $V(G)$.  Now
\begin{itemize}
\item[(a)] $w_G({\P})\ge 5$, and
\item[(b)] $w_G({\P})\ge 8$ if ${\P}$ is normal.
\end{itemize}
\end{claim}
\begin{proof}
Our proof is by contradiction.  For an almost trivial partition ${\P}$, we have
$w_G({\P})\ge w_G(V(G))-2(3)+11\ge 5$, since $G$ does not contain $4K_2$ by
Claim~\ref{CL: nostrongZ5}(a).  If $\P=\{V(G)\}$, then $w_G(\P)=0-11+19=8$.
Since $|G|\ge 4$ by Claim~\ref{CL: nostrongZ5}(b),
all other nontrivial partitions are normal.

Let ${\P}=\{P_1, P_2,\dots, P_t\}$ be a normal partition of $V(G)$.  By
symmetry we assume $|P_1|> 1$ and let $H=G[P_1]$.  For any partition
${\Q}=\{Q_1, Q_2,\dots, Q_s\}$ of $V(H)$, by Eq.~(\ref{EQ: wH}) the refinement
${\Q}\cup ({\P}\setminus\{P_1\})$ of $\P$ satisfies
\begin{eqnarray}\label{EQ: wPwQ}
 w_H({\Q})= w_G({\Q}\cup ({\P}\setminus\{P_1\})) - w_G({\P})+8.
\end{eqnarray}

(a) We first show that $w_G(\P)\ge 5$.  If $w_G(\P)\le 4$, then
Eq.~(\ref{EQ: wPwQ}) implies $w_H(\Q)\ge 4$ for any partition $\Q$ of $H$,
since $w_G({\Q}\cup (\P\setminus\{P_1\}))\ge 0$. Hence $w(H)\ge 4$ and $H$ is
strongly $\Z_5$-connected by Lemma~\ref{LEM: smallH}(c), which contradicts
Claim~\ref{CL: nostrongZ5}.  This proves (a).

(b) 
We now show that $w_G({\P})\ge 8$. Suppose to the
contrary that $w_G({\P})\le 7$. If ${\P}$ contains at least
two nontrivial parts, say $|P_2|>1$, then (a) implies $w_G({\Q}\cup
({\P}\setminus\{P_1\}))\ge 5$ for any partition ${\Q}$
of $H$. Hence $w(H)\ge 6$ by Eq.~(\ref{EQ: wPwQ}), and so $H$ is strongly
$\Z_5$-connected by Lemma~\ref{LEM: smallH}(c), which contradicts Claim~\ref{CL: nostrongZ5}.
So assume instead that ${\P}$ contains a unique nontrivial part $P_1$ and $|P_1|\ge 3$.
For any nontrivial partition ${\Q}$ of $H$, the refinement ${\Q}\cup
({\P}\setminus\{P_1\})$ of $\P$ is a nontrivial partition of $G$, and so
$w_G({\Q}\cup ({\P}\setminus\{P_1\}))\ge 5$ by (a).   Thus $w_H({\Q})\ge 6$ for
any nontrivial partition ${\Q}$ of $H$ by Eq.~(\ref{EQ: wPwQ}). For the trivial
partition ${\Q}^*$ of $H$, since $w_G({\P})\le 7$, Eq.~(\ref{EQ: wPwQ}) implies
$w_H({\Q}^*)\ge 1$.  Since $|H|=|P_1|\ge 3$, we know $H\notin \{2K_2,3K_2\}$.
Since $w(H)\ge 1$, we know $H\not\cong T_{a,b,c}$ with $a+b+c\le 7$.  So Lemma
\ref{LEM: smallH}(a) implies that $H$ is strongly $\Z_5$-connected, which
contradicts Claim~\ref{CL: nostrongZ5}.
\end{proof}

The next two claims are consequences of Claim \ref{CL: nontrivialge4};
they give lower bounds on the edge-connectivity of $G$.

\begin{claim}\label{CL: 2nontrivialge6}
For a partition ${\P}=\{P_1, P_2,\dots, P_t\}$,
\begin{itemize}
\item[(a)] if $|P_1|\ge 2$ and $|P_2|\ge 2$, then  $w({\P})\ge 10$; and
\item[(b)] if $|P_1|\ge 2$ and $|P_2|\ge 3$, then  $w({\P})\ge 13$.
\end{itemize}
\end{claim}
\begin{proof}
Let $H=G[P_1]$ and ${\Q}=\{Q_1, Q_2,\dots, Q_s\}$ be a partition of $H$. Let
$\P'= {\Q}\cup ({\P}\setminus\{P_1\})$.  Note that if $|P_2|\ge 2$, then the
refinement $\P'$ is nontrivial, and if $|P_2|\ge 3$, then $\P'$ is normal.
By Eq.~(\ref{EQ: wH}),
\begin{align*}
w_G(\P') &= w_H({\Q})+ w_G({\P})-8.
\end{align*}

(a) If $w_G(\P)\le 9$, then $w_H(\Q)\ge 4$ for any  partition $\Q$ of $H$ since
$w_G(\P')\ge 5$ by Claim~\ref{CL: nontrivialge4}(a).
So $H$ is strongly $\Z_5$-connected by Lemma \ref{LEM: smallH}(c), which
contradicts Claim \ref{CL: nostrongZ5}.

(b) Similar to (a), if $w_G({\P})\le 12$, then $w_H({\Q})\ge 4$ for any
partition ${\Q}$ of $H$ since $w_G(\P')\ge 8$ by Claim~\ref{CL: nontrivialge4}(b).
Again $H$ is strongly $\Z_5$-connected by Lemma \ref{LEM: smallH}(c), which
contradicts Claim \ref{CL: nostrongZ5}.
\end{proof}

\begin{claim}\label{CL: ess7}
Let $[X, X^c]$ be an edge cut of $G$.
\begin{itemize}
\item[(a)] Now $|[X,X^c]|\ge 6$. That is, $G$ is $6$-edge-connected.
\item[(b)] If $|X|\ge 2$ and $|X^c|\ge 3$, then $|[X,X^c]|\ge 8$.
\end{itemize}
\end{claim}
\begin{proof}
If $[X, X^c]$ is an edge cut of $G$, then $\P=\{X, X^c\}$ is a
partition of $V(G)$.
(a) Clearly ${\cal P}$ is normal, since
$|G|\ge 4$ by Claim~\ref{CL: nostrongZ5}(b). Now Claim \ref{CL:
nontrivialge4}(b) implies $8\le w_G(\P)=2|[X,X^c]|-22+19$, which yields
$|[X,X^c]|\ge 6$.
 (b) If $|X|\ge 2$ and $|X^c|\ge 3$, then $w(\P)\ge 13$ by Claim \ref{CL:
2nontrivialge6}(b).
So $13\le w_G(\P)=2|[X,X^c]|-22+19$, which implies $|[X,X^c]|\ge 8$.
\end{proof}

Next we show that $G$ contains no copy of any graph in Figure~\ref{FIG: W123}
below.  We write $H\dit$\aside{$H\dit$, $H\diit$} to denote the graph formed
from $H$ by subdividing one copy of an edge of maximum multiplicity.  So, for
example, $4K_2\dit=T_{1,1,3}$.  We write $H\diit$ to denote $(H\dit)\dit$.
(The reader may think of the $\circ$ as representing the new 2-vertex.)

\begin{claim}\label{CL: noW1}
$G$ has no copy of $T_{1,1,3}$.
\end{claim}
\begin{proof}
Suppose $G$ contains a copy of $T_{1,1,3}$, with vertices $x,y,z$ and
$\mu(xy)=3$. We lift $xz, zy$ to become a new edge $xy$ and then contract the
corresponding $4K_2$ (contract $xy$). Let $G'$ denote the resulting graph.
The trivial partition ${\cal Q}^*$ of $G'$ satisfies $w_{G'}({\cal Q}^*)\ge
w(G)-2(5)+11\ge 1$. Every nontrivial partition ${\Q'}$ of $G'$
corresponds to a normal partition $\Q$ of $G$ in which the contracted vertex is
replaced by $\{x, y\}$. Since $xz, zy$ are the only two edges possibly counted
in $w_{G}({\Q})$ but not in $w_{G'}({\Q'})$, we have
$w_{G'}({\Q'})\ge w_{G}({\Q})- 4\ge 4$, by Claim~\ref{CL:
nontrivialge4}(b). Thus $w(G')\ge 1$.  By Claim \ref{CL: ess7}, $G$ is
$6$-edge-connected, so $G'$ is $4$-edge-connected.
Thus $G'$ is strongly $\Z_5$-connected, by Lemma \ref{LEM: smallH}(b).
This is a lifting reduction of the first type, so $G$ is strongly
$\Z_5$-connected, which is a contradiction.
\end{proof}

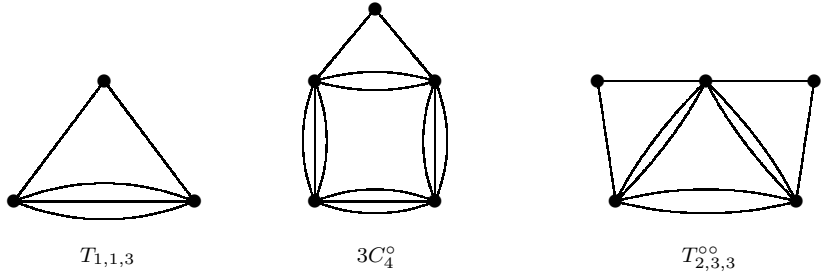
\begin{figure}[t]

\setlength{\unitlength}{0.08cm}

\begin{center}

\begin{picture}(130,40)
\put(0,0){\circle*{2}}\put(30,0){\circle*{2}}\put(15,20){\circle*{2}}
\qbezier(0, 0)(15, 6)(30, 0)\qbezier(0, 0)(15, -6)(30, 0)\qbezier(0, 0)(15, 0)(30, 0)\qbezier(0, 0)(0, 0)(15, 20)\qbezier(30, 0)(15, 20)(15, 20)

\put(50,0){\circle*{2}}\put(70,0){\circle*{2}}\put(50,20){\circle*{2}} \put(70,20){\circle*{2}} \put(60,32){\circle*{2}}

\qbezier(50, 0)(70, 0)(70, 0)\qbezier(50, 0)(60, -4)(70, 0)\qbezier(50, 0)(60, 4)(70, 0)\qbezier(50, 0)(50, 10)(50, 20)\qbezier(50, 0)(46, 10)(50, 20)\qbezier(50, 0)(54, 10)(50, 20)\qbezier(70, 20)(60, 17)(50, 20)\qbezier(70, 20)(60, 23)(50, 20)\qbezier(70, 20)(66, 10)(70, 0)\qbezier(70, 20)(70, 10)(70, 0)\qbezier(70, 20)(74, 10)(70, 0)\qbezier(70, 20)(70, 20)(60, 32)\qbezier(50, 20)(50, 20)(60, 32)

\put(100,0){\circle*{2}}\put(130,0){\circle*{2}}\put(115,20){\circle*{2}} \put(133,20){\circle*{2}}\put(97,20){\circle*{2}}
\qbezier(100, 0)(115, 4)(130, 0)\qbezier(100, 0)(115, -4)(130, 0)\qbezier(100, 0)(105, 10)(115, 20)\qbezier(100, 0)(110, 10)(115, 20)\qbezier(130, 0)(120, 10)(115, 20)\qbezier(130, 0)(125, 10)(115, 20)
\qbezier(100, 0)(98.5, 10)(97, 20)\qbezier(115, 20)(106, 20)(97, 20)\qbezier(133, 20)(124, 20)(115, 20)\qbezier(130, 0)(131.5, 10)(133, 20)

\put(11,-10){\footnotesize{$T_{1,1,3}$}}\put(57,-10){\footnotesize{$3C_4\dit$}}\put(111,-10){\footnotesize{$T_{2,3,3}\diit$}}
\end{picture}
\end{center}
\vspace{0.4cm}
\caption{The graphs $T_{1,1,3}, 3C_4\dit, T_{2,3,3}\diit$.}
\label{FIG: W123}
\end{figure}

\begin{claim}\label{CL: noW2}
$G$ has no copy of $3C_4\dit$.
\end{claim}
\begin{proof}
Suppose $G$ contains a copy of $3C_4\dit$, with vertices $v_1, v_2, v_3, v_4,
z$, where $z$ is a 2-vertex with $N(z)=\{v_1,v_2\}$.
We lift $v_1z, zv_2$ to become a new edge
$v_1v_2$ and then contract the corresponding $3C_4$ to obtain the graph $G'$.
For the trivial partition ${\Q}^*$ of $G'$, we have $w_{G'}({\Q}^*)\ge
w(G)-2(13)+3(11)\ge 7$. For every nontrivial partition ${\Q'}$ of $G'$,  we have
$w_{G'}({\Q'})\ge w_{G}({\Q})- 4\ge 4$ for the same reason as in the previous claim.
Thus $w(G')\ge 4$, so $G'$ is strongly 
$\Z_5$-connected by Lemma~\ref{LEM: smallH}(c).
This is a lifting reduction of the first type.
Hence $G$ is strongly $\Z_5$-connected, which contradicts Claim~\ref{CL:
nostrongZ5}.
\end{proof}

Now we can slightly strengthen Claim~\ref{CL: nontrivialge4}(b).

\begin{claim}
\label{CL: noalmostge9}
Every normal partition ${\P}=\{P_1, P_2,\dots, P_t\}$ satisfies
  $$w({\P})\ge 9.$$
\end{claim}
\begin{proof}
  Let ${\P}=\{P_1, P_2,\dots, P_t\}$ be a normal partition of $G$ with
$|P_1|>1$. Suppose to the contrary that $w({\P})=8$, by Claim~\ref{CL:
nontrivialge4}(b). Now $|P_1|\ge 3$ and $|P_2|=\ldots=|P_t|=1$, by
Claim~\ref{CL: 2nontrivialge6}(a). As in Claim~\ref{CL: nontrivialge4},
let $H=G[P_1]$, let ${\Q}=\{Q_1, Q_2,\dots, Q_s\}$ be a partition of $H$, and
let $\P'={\Q}\cup ({\P}\setminus\{P_1\})$ be a refinement of $\P$.
Eq.~(\ref{EQ: wH}) implies
\begin{eqnarray*}
 w_H({\Q})= w_G(\P') - w_G({\P})+8=w_G(\P').
\end{eqnarray*}
If ${\Q}$ is a nontrivial partition of $H$, then $\P'$ is nontrivial in $G$, so
$w_H({\Q})= w_G(\P')\ge 5$, by Claim~\ref{CL: nontrivialge4}(a). If $\Q$ is
the trivial partition of $H$, then $w_H(\Q)= w_G(\P')\ge 0$.  Since
$|H|=|P_1|\ge 3$, we know $H\notin\{2K_2,3K_2\}$.  And since $G$ has no copy of
$T_{1,1,3}$, by Claim \ref{CL: noW1}, we know $H\notin\{T_{1,3,3}, T_{2,2,3}\}$.
Now Lemma \ref{LEM: smallH}(a) implies that $H$ is strongly $\Z_5$-connected,
which contradicts Claim~\ref{CL: nostrongZ5}.
\end{proof}

Claim~\ref{CL: noalmostge9} allows us to also prove that the third graph
in Figure~\ref{FIG: W123} is reducible.

\begin{claim}\label{CL: noW3}
$G$ has no copy of $T_{2,3,3}\diit$.
\end{claim}
\begin{proof}
Suppose $G$ contains a copy of $T_{2,3,3}\diit$ with vertices $w,x,y,z_1,z_2$, where
$z_1$ and $z_2$ are 2-vertices with $N(z_1)=\{w,x\}$ and $N(z_2)=\{x,y\}$.
We lift $wz_1, z_1x$ to become a new edge $wx$,
and lift $xz_2, z_2y$ to become a new edge $xy$.  Now $\{w,x,y\}$ induces a copy
of $T_{2,3,3}$, so we contract $\{w,x,y\}$ to form a graph $G'$. Since
$\delta(G)\ge 6$ by Claim~\ref{CL: ess7}(a), we have $\delta(G')\ge 4$.  The size
of each edge cut decreases at most $4$ from $G$ to $G'$, and it decreases at
least $3$ only if that edge cut has at least two vertices on each side.
In that case Claim~\ref{CL: ess7}(b) shows the original edge cut in $G$ has
size at least $8$.  Since $G$ is $6$-edge-connected by Claim~\ref{CL: ess7},
each edge cut in $G'$ has size at least $4$, so $G'$ is $4$-edge-connected.

The trivial partition ${\Q}^*$ of $G'$ satisfies $w_{G'}({\Q}^*)\ge
w(G)-2(10)+11(2)\ge 2$. Every nontrivial partition $\Q'$ of $G'$
corresponds to a normal partition $\Q$ of $G$ in which the
contracted vertex is replaced by $\{w, x, y\}$.  So $w_{G'}({\Q'})\ge
w_{G}({\Q})-2(4)\ge 1$, by Claim \ref{CL: noalmostge9}. Thus, $G'$ is
$4$-edge-connected and $w(G')\ge 1$. By Lemma \ref{LEM: smallH}(b), $G'$ is
strongly $\Z_5$-connected. This is a lifting reduction of the first type.
Since $T_{2,3,3}$ is strongly $\Z_5$-connected by Lemma~\ref{LEM:
4K2J2K4inSZ5}, graph $G$ is strongly $\Z_5$-connected, which
contradicts Claim~\ref{CL: nostrongZ5}.
\end{proof}

\subsection{The final step: Discharging}
Now we use discharging to show that some subgraph in Figure~\ref{FIG:
4K2J2K4C4} or \ref{FIG: W123} must appear in $G$.  This contradicts one of
the claims in the previous section, and thus finishes the proof.

Fix a plane embedding of $G$.  (We assume that all parallel edges between two
vertices $v$ and $w$ are embedded consecutively, in the cyclic orders, around
both $v$ and $w$.)
Let $F(G)$ denote the set of all faces of $G$. For each face $f\in F(G)$, we
write $\ell(f)$ for its length.  A face $f$ is  a \emph{$k$-face},
\emph{$k^+$-face}, or \emph{$k^-$-face}\aaside{$k/k^+/k^-$-face}{-4mm} if
(respectively) $\ell(f)=k$, $\ell(f)\ge k$, or $\ell(f)\le k$.
A sequence of faces $f_1f_2\ldots f_s$ is called a \Emph{face chain} if, for
each $i\in\{1,\ldots,s-1\}$, faces $f_i$ and $f_{i+1}$ are adjacent, i.e., their
boundaries share a common edge. The \emph{length} of this chain is $s+1$.
Two faces $f$ and $f'$ are \Emph{weakly adjacent} if there is a face chain
$ff_1\ldots f_s f'$ such that that $f_i$ is a $2$-face for each
$i\in\{1,\ldots,s\}$. We allow $s$ to be $0$, meaning
$f$ and $f'$ are adjacent.
A \EmphE{string}{4mm} is a maximal face chain such that each of its faces is a
$2$-face. The boundary of a string consists of two edges, each of which is
incident to a $3^+$-face.
A $k$-face is called a $(t_1, t_2, \ldots, t_k)$-face if its boundary edges are
contained in strings with lengths $t_1, t_2, \ldots, t_k$. Here $t_i$
is allowed to be $1$, meaning the corresponding edge is not contained in a string.

Since $w(G)\ge 0$, we have $2\|G\|-11|G|+19\ge 0.$
By Euler's Formula, $|G|+|F(G)|-\|G\|=2$.  We solve for $|G|$ in the
equation and substitute into the inequality, which gives

\begin{align}
\sum_{f\in F(G)}\ell(f)=2\|G\|\le \frac{22}{9}|F(G)|-\frac{2}{3}.
\label{EQ: totalcharge}
\end{align}

We assign to each face $f$ initial charge $\ell(f)$. So the total charge is
strictly less than $22|F(G)|/9$.
To redistribute charge, we use the following three discharging rules.
\begin{itemize}
\item[(R1)] Each $2$-face receives charge $\frac{2}{9}$ from each weakly
adjacent $3^+$-face.
\item[(R2)] Each $(2,2,2)$-face receives charge $\frac{1}{9}$ from each weakly
adjacent $4^+$-face and $(2,1,1)$-face.
\item[(R3)] Each $(2,2,2)$-face receives charge $\frac{1}{18}$ from each weakly adjacent $(2,2,1)$-face.
\end{itemize}

If two faces are weakly adjacent through multiple edges or strings,
then the discharging rules apply for each edge and string.
After applying these rules, we claim that every face has charge at least
$\frac{22}{9}$, which contradicts Eq.~\eqref{EQ: totalcharge}.

Each $2$-face ends with $2+2(\frac29)=\frac{22}9$.
Since $G$ contains no $4K_2$ and no $T_{1,1,3}$, the charge each face sends
across each boundary edge is at most $2(\frac29)$.  Thus, when $k\ge 5$ each
$k$-face ends with at least $k-k(2(\frac{2}{9}))=\frac{5k}9\ge \frac{25}{9}$.
Since $G$ contains no $3C_4$ and no $3C_4\dit$, each $4$-face ends with at least
$4-7(\frac{2}{9})=\frac{22}{9}$. It is straightforward to check that each
$(1,1,1)$-face ends with $3$, each $(2,1,1)$-face ends with at least
$3-\frac{2}{9}-\frac{1}{9}=\frac{24}{9}$, and each $(2,2,1)$-face ends with
at least $3-2(\frac{2}{9})-2(\frac1{18})=\frac{22}{9}$.
It remains to check $(2,2,2)$-faces.

Suppose to the contrary that a $(2,2,2)$-face $xyz$ ends with less than
$\frac{22}{9}$. After (R1), face $xyz$ has $3-3(\frac29)=\frac{21}9$.  Since
$xyz$ ends with less than $\frac{22}9$, it receives at most $\frac1{18}$ by
(R2) and (R3). So $xyz$ must be adjacent to
three $3$-faces, and at most one of these is a $(2,2,1)$-face, while the
others are $(2,2,2)$-faces. By Claim~\ref{CL: noW3}, $G$ contains no
$T_{2,3,3}\diit$, so the three $3$-faces adjacent to $xyz$ must share a new
common vertex, say $w$. If one of $wx, wy, wz$ is not contained in a string,
then $xyz$ is adjacent to two $(2,2,1)$-faces, and so receives at least
$2(\frac1{18})$ by (R3), contradicting our assumption above.
Thus we assume $\mu(wx)=\mu(wy)=\mu(wz)=2$.  So $G[\{x,y,z,w\}]$ contains a
$2K_4$, contradicting Claim~\ref{CL: nostrongZ5}(a).
This shows that each $(2,2,2)$-face ends with at least $\frac{22}{9}$, which
completes the proof.

\section{Circular $7/3$-flows: Proof of Theorem~\ref{7/3-flow-thm}}
\label{Z7-sec}
In this section we prove Theorem~\ref{7/3-flow-thm}.
As in the previous section, this theorem is implied by the more technical
result, Theorem~\ref{THM: Main2}.  The proof of Theorem~\ref{THM: Main2}
is similar to that of Theorem~\ref{THM: Main1}, but with more reducible
configurations and more details.

\subsection{Preliminaries on Modulo $7$-orientations}

We define a weight function $\rho$ as follows (which is similar to $w$ in
Definition~\ref{DEF: partition}).
\begin{definition}\label{DEF: rho-partition}
Let ${\P}=\{P_1, P_2,\dots, P_t\}$ be a partition of $V(G)$.
Let $$\rho_G({\P})=\sum_{i=1}^{t}d(P_i)-17t+31$$
and $\rho(G)=\min\{\rho_G({\P}): {\P}\ is\ a\ partition\ of\ V(G)\}.$
\end{definition}

Analogous to Lemma~\ref{mod5-key-lem}, we have the following.

\begin{lemma}
  Let ${\P}=\{P_1, P_2,\dots, P_t\}$ be a partition of $V(G)$ with $|P_1|>1$.
Let $H=G[P_1]$ and let ${\Q}=\{Q_1, Q_2,\dots, Q_s\}$ be a partition of
$V(H)$.  Now ${\Q}\cup ({\P}\setminus\{P_1\})$ is a refinement of $\P$
satisfying
\begin{align}\label{EQ: wH2}
  \rho_G({\Q}\cup ({\P}\setminus\{P_1\}))= \rho_H({\Q})+ \rho_G({\P})-(31-17).
\end{align}
\label{mod7-key-lem}
\end{lemma}

\begin{proof}
The proof is identical to that of Lemma~\ref{mod5-key-lem}, with 17 in place
of 11 and with 31 in place of 19.
\end{proof}

We typically assume that each edge has multiplicity at most 5 (since
$6K_2$ is strongly $\Z_7$-connected, and so cannot appear in a minimal
counterexample to Theorem~\ref{THM: Main2}, as we prove in Claim~\ref{CL:
nostrongZ7}, below).  Now $\rho(aK_2)=2a-3$,
$\rho(T_{a,b,c})=2a+2b+2c-20$,  and $\rho(3K_4)=-1$; see Figure~\ref{FIG:
K234}.  In each case, the minimum in the definition of $\rho$ is achieved
uniquely by the partition with each vertex in its own part.

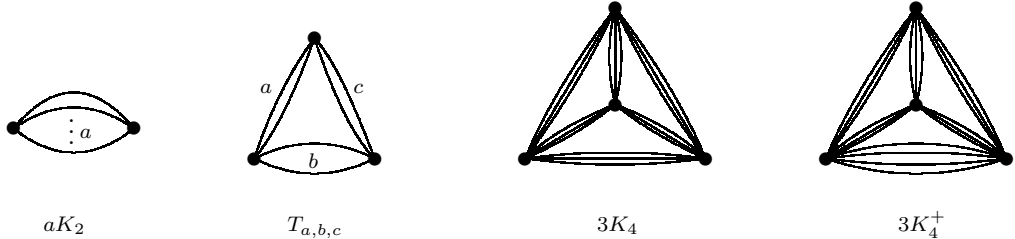
\begin{figure}[ht]

\setlength{\unitlength}{0.08cm}

\begin{center}

\begin{picture}(170,35)
\put(0,10){\circle*{2}}\put(20,10){\circle*{2}}
\qbezier(0, 10)(10, 22)(20, 10)
\qbezier(0, 10)(10, 17)(20, 10)\qbezier(0, 10)(10, 2)(20, 10)
\put(9,7.5){$\vdots$}\put(11,8.5){\footnotesize{$a$}}\put(5,-7){\footnotesize{$aK_2$}}


\put(40,5){\circle*{2}}\put(60,5){\circle*{2}}\put(50,25){\circle*{2}}
\qbezier(40, 5)(50, 10)(60, 5)
\qbezier(40, 5)(50, 0)(60, 5)
\qbezier(40, 5)(45, 11)(50, 25)\qbezier(40, 5)(44, 17)(50, 25)
\qbezier(60, 5)(55, 10)(50, 25)\qbezier(60, 5)(54, 21)(50, 25)
\put(45.5,-7){\footnotesize{$T_{a,b,c}$}}
\put(49,3.5){\footnotesize{$b$}}\put(41,16){\footnotesize{$a$}}\put(56.5,16){\footnotesize{$c$}}

\put(85,5){\circle*{2}}\put(115,5){\circle*{2}}\put(100,30){\circle*{2}}\put(100,14){\circle*{2}}
\qbezier(85, 5)(100, 7)(115, 5)\qbezier(85, 5)(100, 3)(115, 5)\qbezier(85, 5)(85, 5)(115, 5)
\qbezier(85, 5)(90, 17)(100, 30)\qbezier(85, 5)(94, 17)(100, 30)\qbezier(85, 5)(85, 5)(100, 30)
\qbezier(115, 5)(110, 17)(100, 30)\qbezier(115, 5)(106, 17)(100, 30)\qbezier(115, 5)(115, 5)(100, 30)
\qbezier(85, 5)(90, 10)(100, 14)\qbezier(85, 5)(94, 9)(100, 14)\qbezier(85, 5)(85, 5)(100, 14)
\qbezier(115, 5)(110, 10)(100, 14)\qbezier(115, 5)(106, 9)(100, 14)\qbezier(115, 5)(115, 5)(100, 14)
\qbezier(100, 30)(98, 22)(100, 14)\qbezier(100,30)(102, 22)(100, 14)\qbezier(100,30)(100, 30)(100, 14)

\put(97,-7){\footnotesize{$3K_4$}}

\put(135,5){\circle*{2}}\put(165,5){\circle*{2}}\put(150,30){\circle*{2}}\put(150,14){\circle*{2}}
\qbezier(135, 5)(150, 10)(165, 5)\qbezier(135, 5)(150, 7)(165, 5)\qbezier(135, 5)(150, 3)(165, 5)
\qbezier(135, 5)(150, 0)(165, 5)

\qbezier(135, 5)(140, 17)(150, 30)\qbezier(135, 5)(144, 17)(150, 30)\qbezier(135, 5)(135, 5)(150, 30)
\qbezier(165, 5)(160, 17)(150, 30)\qbezier(165, 5)(156, 17)(150, 30)\qbezier(165, 5)(165, 5)(150, 30)
\qbezier(135, 5)(140, 10)(150, 14)\qbezier(135, 5)(144, 9)(150, 14)\qbezier(135, 5)(135, 5)(150, 14)
\qbezier(165, 5)(160, 10)(150, 14)\qbezier(165, 5)(156, 9)(150, 14)\qbezier(165, 5)(165, 5)(150, 14)
\qbezier(150, 30)(148, 22)(150, 14)\qbezier(150,30)(152, 22)(150, 14)\qbezier(150,30)(150, 30)(150, 14)

\put(147,-7){\footnotesize{$3K_4^+$}}

\end{picture}
\end{center}
\vspace{0.4cm}
\caption{\small\it The graphs $aK_2, T_{a,b,c}, 3K_4, 3K_4^+$.}
\label{FIG: K234}
\end{figure}

Let
${\F}=\{aK_2: 2\le a\le 5\}\cup\{T_{a,b,c}: 10\le a+b+c\le 11
~\mbox{and $T_{a,b,c}$ is 6-edge-connected}.\}$  
It is straightforward\footnote{When $a\le 5$, the graph $aK_2$ has seven $\Z_7$-boundaries and
at most 6 orientations, so at least one boundary is not achievable.  The
graph $3K_4$ cannot achieve the boundary $\beta(v)=0$ for all $v$.  In such an
orientation $D$ each vertex $v$ must have $|d^+_D(v)-d^-_D(v)|=7$.  But now
some two adjacent vertices must either both have indegree 8 or both have
outdegree 8, and we cannot orient the three edges between them to achieve this.
For $T_{a,b,c}$, it suffices to consider the case $a+b+c=11$.  Let
$V(G)=\{v_1,v_2,v_3\}$.  By symmetry, we assume $d(v_1)\le d(v_2)\le d(v_3)$.
For $T_{1,5,5}$, we cannot achieve $\beta(v_1)=\beta(v_2)=1$ and $\beta(v_3)=5$,
since $v_1$ and $v_2$ must each have all incident edges oriented in.
For $T_{2,4,5}$, we cannot achieve $\beta(v_1)=1$, $\beta(v_2)=2$, and $\beta(v_3)=4$,
since $v_1$ must have all incident edges oriented in, and $v_2$ must have all
but one edges oriented in.
For $T_{3,3,5}$, we cannot achieve $\beta(v_1)=1$ and $\beta(v_2)=\beta(v_3)=3$,
since $v_1$ must have all incident edges oriented in.
For $T_{3,4,4}$, we cannot achieve $\beta(v_1)=\beta(v_2)=2$ and $\beta(v_3)=3$,
since $v_1$ and $v_2$ must each have all but one incident edge oriented in.
}
to check that no graph in $\F$ is strongly $\Z_7$-connected.
Further, if $T_{a,b,c}$ is 8-edge-connected, then $\|G\|\ge 3\delta(G)/2\ge 12$.
Thus, no graph in $\F$ is 8-edge-connected.
The following theorem is the main result of Section~\ref{Z7-sec}.   We call a
partition $\P$ \EmphE{problematic}{-4mm} if $G/\P\in \F$.

\begin{theorem}
\label{THM: Main2}
Let $G$ be a planar graph and $\beta$ be a $\Z_7$-boundary of $G$. If $\rho(G)\ge
0$, then $G$ admits a $(\Z_7,\beta)$-orientation, unless $G$ has a problematic partition.
\end{theorem}

As easy corollaries of Theorem~\ref{THM: Main2} we get the following two
results.

\begin{theorem}
Every $17$-edge-connected planar graph is strongly $\Z_7$-connected.
\label{17-edge-thm}
\end{theorem}

\begin{theorem}
Every odd-$17$-edge-connected planar graph admits a modulo $7$-orientation.  In
particular, every 16-edge-connected planar graph admits a modulo
$7$-orientation (and thus a circular 7/3-flow).
\label{16-edge-thm}
\end{theorem}

The proofs of Theorems~\ref{17-edge-thm} and~\ref{16-edge-thm} are identical to
those of Theorems~\ref{11-edge-thm} and~\ref{10-edge-thm}, but with 17 in place
of 11 and with 31 in place of 19.  Note that Theorem~\ref{16-edge-thm} includes
Theorem~\ref{7/3-flow-thm} as a special case.

For the proof of Theorem~\ref{THM: Main2}, we need the following two lemmas.
Their proofs are more tedious than enlightening, so we postpone them to the appendix.
When a graph $H$ is edge-transitive, we write $H^+$ or $H^-$\aside{$H^+/H^-$}
to denote the graph formed by adding or removing a single copy of one edge.

\begin{lemma}
\label{Z7-contract-configs}
Each of the following graphs is strongly $\Z_7$-connected: $6K_2$,
$3K_4^{+}$, and every 6-edge-connected graph $T_{a,b,c}$ where $a+b+c=12$.
\end{lemma}

Let \Emph{$5C_4^=$} denote the graph formed from $5C_4$ by deleting a perfect matching.
\begin{lem}
The graph $5C_4^=$ is strongly $\Z_7$-connected.  Further, if $G$ is a graph
with $|G|=4$, $\|G\|=19$, $\mu(G)\le 5$, and $\delta(G)\ge 8$, then $G$ is
strongly $\Z_7$-connected.
\label{K4weightsZ7-lem}
\end{lem}

\subsection{Properties of a Minimal Counterexample in Theorem \ref{THM: Main2}}

{\bf Let $G$ be a counterexample to Theorem \ref{THM: Main2}  that minimizes
$|G|+\|G\|$.} Thus Theorem~\ref{THM: Main2} holds for all graphs smaller
than $G$. This implies the following lemma, which we will use frequently.
\begin{lemma}
\label{LEM: smallH2}
If $H$ is a planar graph with $\rho(H)\ge 0$ and $|H|+\|H\|<|G|+\|G\|$, then each of the following holds.
\begin{itemize}
\item[(a)] If $\rho_H({\P})\ge 8$ for every nontrivial partition ${\P}$,
then $H$ is strongly $\Z_7$-connected unless $H\in {\F}$.
\item[(b)] If $\rho(H)\ge 8$, then $H$ is strongly $\Z_7$-connected.
\item[(c)] Assume that $H$ is $6$-edge-connected.
   \begin{itemize}
     \item[(c-i)] If $\rho_H({\P})\ge 3$ for every nontrivial partition
${\P}$, then $H$ is strongly $\Z_7$-connected unless $H\cong T_{a,b,c}$
with $a+b+c\in\{10,11\}$.
     \item[(c-ii)] If $\rho(H)\ge 3$, then $H$ is strongly $\Z_7$-connected.
     \item[(c-iii)] If $H$ is $8$-edge-connected, then  $H$ is strongly $\Z_7$-connected.
   \end{itemize}
\end{itemize}
\end{lemma}
\begin{proof}
We apply Theorem \ref{THM: Main2} to $H$.
(a) For each $J\in \F$, the trivial partition $\Q^*$ satisfies
$\rho_J(\Q^*)\le \max\{2(5)-2(17)+31,2(11)-3(17)+31\}=7$.  Since $\rho_H(\P)\ge 8$
for every nontrivial partition $\P$, we know that $H/\P\notin \F$.
Part (b) follows immediately from (a).  Consider (c).  Since $H$ is
6-edge-connected, there does not exist $\P$ such that $|H/\P|=2$ and
$\|H/\P\|\le 5$.  For (c-i), suppose there is a nontrivial partition $\P$ such
that $H/\P\cong T_{a,b,c}$ with $a+b+c\in\{10,11\}$.  Now
$\rho_H(\P)=2(11)-3(17)+31=2$, which contradicts the hypothesis.  Note that
(c-ii) follows directly from (c-i).
Finally, we prove (c-iii).
Since $G$ is 8-edge-connected, so is $G/\P$, for each partition $\P$.
Recall that each element of $\F$ has edge-connectivity at most 7.
Thus, $G/\P\notin \F$.
\end{proof}

As in Section~\ref{Z5-sec}, the main idea of the proof is to show that
$\rho_G(\P)$ is relatively large for each nontrivial partition $\P$.  This gives
us the ability to apply Lemma~\ref{LEM: smallH2} to subgraphs of $G$ even after
modifying them slightly, which yields more power when proving subgraphs are
reducible.
\begin{claim}
\label{CL: nostrongZ7}
$G$ has no strongly $\Z_7$-connected subgraph $H$ with $|H|>1$. In particular,
\begin{itemize}
\item[(a)] $G$ has no copy of $6K_2$, $3K_4^+$, or
a 6-edge-connected graph $T_{a,b,c}$ with $a+b+c=12$; and
\item[(b)] $|G|\ge 4$.
\end{itemize}
\end{claim}
\begin{proof}
The proof of the first statement is identical to that of
Claim~\ref{CL: nostrongZ5}, with $\Z_7$ in place of $\Z_5$.  Note that (a) follows from the first statement and
Lemma~\ref{Z7-contract-configs}.

Now we prove (b). Clearly $|G|\ge 2$, so
first suppose $|G|=2$. Since $\rho(G)\ge 0$, we know $\|G\|\ge 2$.
Since $G$ has no problematic partition, we know $\|G\|\ge 6$.  But now $G$ contains
$6K_2$, which contradicts (a).  So assume $|G|=3$, that is $G=T_{a,b,c}$.  Since
$\rho(G)\ge 0$, we know $a+b+c\ge 10$.  Since $G$ has no problematic partition,
$G$ is 6-edge-connected.  By the definition of $\F$, this implies that $a+b+c\ge
12$.  Recall that $G$ contains no $6K_2$ by (a); thus $\max\{a,b,c\}\le 5$.  A
short case analysis shows that $G$ contains as a subgraph one of $T_{2,5,5}$,
$T_{3,4,5}$, or $T_{4,4,4}$.  Each of these has 12 edges and is
6-edge-connected, which contradicts (a).
\end{proof}

\begin{claim}
\label{CL: nontrivialge716}
If ${\P}=\{P_1, P_2,\dots, P_t\}$ is a nontrivial partition of $V(G)$, then
\begin{enumerate}
\item[(a)] $\rho_G({\P})\ge 7$; and
\item[(b)] $\rho_G({\P})\ge 12$ if ${\P}$ is normal.
\end{enumerate}
\end{claim}
\begin{proof}
We argue by contradiction.  For an almost trivial partition ${\P}$, we have
$\rho_G({\P})\ge \rho_G(V(G))-2(5)+17\ge 7$, since $G$ does not contain $6K_2$ by
Claim~\ref{CL: nostrongZ7}.  If $\P=\{V(G)\}$, then $w_G(\P)=0-17+31=14$.
Since $|G|\ge 4$ by Claim~\ref{CL: nostrongZ7}(b),
we now only need to consider the weight of normal partitions.

Let ${\P}=\{P_1, P_2,\dots, P_t\}$ be a normal partition of $V(G)$.
We may assume $|P_1|> 1$ and let $H=G[P_1]$.
For any partition ${\Q}=\{Q_1, Q_2,\dots, Q_s\}$ of $V(H)$, by Eq.~(\ref{EQ:
wH2}) the refinement ${\Q}\cup ({\P}\setminus\{P_1\})$ of $\P$ satisfies
\begin{eqnarray}\label{EQ: rPrQ}
 \rho_H({\Q})= \rho_G({\Q}\cup ({\P}\setminus\{P_1\})) - \rho_G({\P})+14.
\end{eqnarray}

(a) We first show that $\rho_G({\P})\ge 7$.  If $\rho_G({\P})\le 6$, then
Eq.~(\ref{EQ: rPrQ}) implies that $\rho_H({\Q})\ge 8$ for any partition
${\Q}$ of $H$, since $\rho_G({\Q}\cup ({\P}\setminus\{P_1\}))\ge 0$. Hence
$\rho(H)\ge 8$ and $H$ is strongly $\Z_7$-connected by Lemma~\ref{LEM:
smallH2}(b), which contradicts Claim~\ref{CL: nostrongZ7}. This proves (a).

(b) We now show that $\rho_G({\P})\ge 12$. Suppose, to the
contrary, that $\rho_G({\P})\le 11$. If ${\P}$ contains at least
two nontrivial parts, say $|P_2|>1$, then (a) implies $\rho_G({\Q}\cup
({\P}\setminus\{P_1\}))\ge 7$ for any partition ${\Q}$
of $H$. Hence $\rho(H)\ge 10$ by Eq.~(\ref{EQ: rPrQ}), and so $H$ is strongly
$\Z_7$-connected by Lemma \ref{LEM: smallH2}(b),  which contradicts Claim~\ref{CL: nostrongZ7}.
Assume instead that ${\P}$ contains a unique nontrivial part $P_1$ and $|P_1|\ge 3$.
For any nontrivial partition ${\Q}$ of $H$, the refinement ${\Q}\cup
({\P}\setminus\{P_1\})$ of $\P$ is a nontrivial partition of $G$, and so
$\rho_G({\Q}\cup ({\P}\setminus\{P_1\}))\ge  7$ by (a).   Thus
$\rho_H({\Q})\ge 10$ for any nontrivial partition ${\Q}$ of $H$
by Eq.~(\ref{EQ: rPrQ}). For the trivial partition ${\Q}^*$ of $H$, since
$\rho_G({\P})\le 11$, Eq.~(\ref{EQ: rPrQ}) implies $\rho_H({\Q}^*)\ge 3$.
Since $|H|=|P_1|\ge 3$, we know $H\not\cong aK_2$.  Since $\rho(H)\ge 3$, we
know $H\not\cong T_{a,b,c}$ with $a+b+c\le 11$.  So Lemma \ref{LEM: smallH2}(a)
implies that $H$ is strongly $\Z_7$-connected, which contradicts Claim \ref{CL:
nostrongZ7}.
\end{proof}

The next two claims follow from Claim \ref{CL: nontrivialge716}.
They give lower bounds on the edge-connectivity of $G$.

\begin{claim}
\label{CL: 2nontrivialge14z7}
For a partition ${\P}=\{P_1, P_2,\dots, P_t\}$,
\begin{enumerate}
\item[(a)] if $|P_1|\ge 2$ and $|P_2|\ge 2$, then  $\rho({\P})\ge 14$; and
\item[(b)] if $|P_1|\ge 2$ and $|P_2|\ge 3$, then  $\rho({\P})\ge 19$.
\end{enumerate}
\end{claim}
\begin{proof}
Let $H=G[P_1]$ and ${\Q}=\{Q_1, Q_2,\dots, Q_s\}$ be a partition of $H$. By
Eq.~(\ref{EQ: wH2}),
\begin{eqnarray}\nonumber
 \rho_H({\Q})= \rho_G({\Q}\cup ({\P}\setminus\{P_1\})) - \rho_G({\P})+14.
\end{eqnarray}

(a) If $\rho_G({\P})\le 13$, then $\rho_H({\Q})\ge 8$ for any partition ${\Q}$ of
$H$ since $\rho_G({\Q}\cup ({\P}\setminus\{P_1\}))\ge 7$ by Claim \ref{CL:
nontrivialge716}(a).  So $H$ is strongly $\Z_5$-connected by Lemma~\ref{LEM:
smallH2}(b), which contradicts Claim~\ref{CL: nostrongZ7}.

(b) Similarly, if $\rho_G({\P})\le 18$, then $\rho_H({\Q})\ge 8$ for any
partition ${\Q}$ of $H$ since $\rho_G({\Q}\cup ({\P}\setminus\{P_1\}))\ge 12$ by
Claim \ref{CL: nontrivialge716}(b).  Again $H$ is strongly $\Z_5$-connected by
Lemma~\ref{LEM: smallH2}(b), which contradicts Claim~\ref{CL: nostrongZ7}.
\end{proof}

\begin{claim}\label{CL: ess11}
Let  $[X, X^c]$ be an edge cut of $G$.
\begin{enumerate}
\item[(a)] Now $|[X,X^c]|\ge 8$. That is, $G$ is $8$-edge-connected.
\item[(b)] If $|X|\ge 2$ and $|X^c|\ge 3$, then $|[X,X^c]|\ge 11$.
\end{enumerate}
\end{claim}
\begin{proof}
(a) Let $\P=\{X, X^c\}$. 
Since $|G|\ge 4$ by Claim \ref{CL: nostrongZ7}(b), the partition $\P$ is normal.
Now Claim~\ref{CL: nontrivialge716}(b) gives
$12\le \rho_G(\P)=2|[X,X^c]|-34+31$, which implies $|[X,X^c]|\ge 8$.

(b) If $|X|\ge 2$ and $|X^c|\ge 3$, then $\rho_G(\P)\ge 19$ by Claim~\ref{CL: 2nontrivialge14z7}(b).
So $19\le \rho_G({\cal P})=2|[X,X^c]|-34+31$, which implies $|[X,X^c]|\ge 11$.
\end{proof}

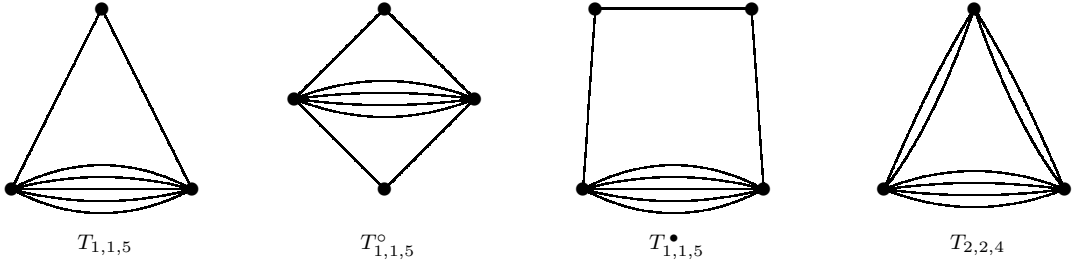
\begin{figure}[ht]

\setlength{\unitlength}{0.08cm}

\begin{center}

\begin{picture}(190,40)

\put(5,0){\circle*{2}}\put(35,0){\circle*{2}}\put(20,30){\circle*{2}}
\qbezier(5, 0)(20, 4)(35, 0)\qbezier(5, 0)(20, -4)(35, 0)\qbezier(5, 0)(20, 0)(35, 0)
\qbezier(5, 0)(20, 8)(35, 0)\qbezier(5, 0)(20, -8)(35, 0)
\qbezier(5, 0)(5, 0)(20, 30)\qbezier(35, 0)(35, 0)(20, 30)

\put(52,15){\circle*{2}}\put(82,15){\circle*{2}}\put(67,30){\circle*{2}}\put(67,0){\circle*{2}}
\qbezier(67, 0)(67, 0)(52, 15)\qbezier(67, 0)(67, 0)(82, 15)\qbezier(67,
30)(67, 30)(52, 15)\qbezier(67, 30)(67, 30)(82, 15)
\qbezier(52, 15)(67, 21)(82, 15)\qbezier(52, 15)(67, 17)(82, 15)\qbezier(52,
15)(67, 13)(82, 15)\qbezier(52, 15)(67, 9)(82, 15)

\put(100,0){\circle*{2}}\put(130,0){\circle*{2}}\put(102,30){\circle*{2}}\put(128,30){\circle*{2}}
\qbezier(100, 0)(115, 4)(130, 0)\qbezier(100, 0)(115, -4)(130, 0)\qbezier(100, 0)(115, 0)(130, 0)
\qbezier(100, 0)(115, 8)(130, 0)\qbezier(100, 0)(115, -8)(130, 0)
\qbezier(100, 0)(100, 0)(102, 30)\qbezier(130, 0)(130, 0)(128, 30)\qbezier(102, 30)(102, 30)(128, 30)

\put(150,0){\circle*{2}}\put(180,0){\circle*{2}}\put(165,30){\circle*{2}}
\qbezier(150, 0)(165, 2)(180, 0)\qbezier(150, 0)(165, -2)(180, 0)\qbezier(150, 0)(165, 6)(180, 0)
\qbezier(150, 0)(165, -6)(180, 0)
\qbezier(150, 0)(158, 19)(165, 30)\qbezier(180, 0)(172, 19)(165, 30)
\qbezier(150, 0)(158, 10)(165, 30)\qbezier(180, 0)(172, 10)(165, 30)

\put(16,-10){\footnotesize{$T_{1,1,5}$}}\put(111,-10){\footnotesize{$T^{\,\bullet}_{1,1,5}$}}\put(63,-10){\footnotesize{$T_{1,1,5}\dit$}}\put(161,-10){\footnotesize{$T_{2,2,4}$}}
\end{picture}
\end{center}
\vspace{0.6cm}
\caption{\small\it The graphs $T_{1,1,5}, T^{\,\bullet}_{1,1,5},T_{1,1,5}\dit, T_{2,2,4}$.
}
\label{FIG: Y123}
\end{figure}

Let \Emph{$T^{\,\bullet}_{1,1,5}$} denote the graph formed from $T_{1,1,5}$ by
subdividing an edge of multiplicity 1. %
We now show that $G$ contains none of the folllowing (shown in Figure~\ref{FIG:
Y123}) as subgraphs: $T_{1,1,5}$, $T_{1,1,5}\dit$,
$T^{\,\bullet}_{1,1,5}$, and
$T_{2,2,4}$.

\begin{claim}\label{CL: noY1}
 $G$ has no copy of $T_{1,1,5}$.
\end{claim}
\begin{proof}
Suppose $G$ contains a copy of $T_{1,1,5}$ with vertices $x,y,z$ and
$\mu(xy)=5$. We lift $xz, zy$ to become a new edge $xy$ and contract the
resulting $6K_2$ induced by $\{x,y\}$.  Let $G'$ denote the resulting graph.
The trivial partition ${\Q}^*$ of $G'$ satisfies $\rho_{G'}({\Q}^*)\ge
\rho(G)-2(7)+17\ge 3$. Every nontrivial partition ${\Q'}$ of $G'$
corresponds to a normal partition $\Q$ of $G$ in which the contracted vertex is
replaced by $\{x, y\}$. Since $xz, zy$ are the only two edges possibly
counted in $\rho_{G}({\Q})$ but not in $\rho_{G'}({\Q'})$, we have
$\rho_{G'}({\Q'})\ge \rho_{G}({\Q})- 2(2)\ge 8$, by Claim~\ref{CL:
nontrivialge716}(b). So $\rho(G')\ge 3$. Since $G$ is $8$-edge-connected by
Claim~\ref{CL: ess11}, graph $G'$ is $6$-edge-connected, and so
$G'$ is strongly $\Z_7$-connected by Lemma~\ref{LEM: smallH2}(c-ii).
This is a lifting reduction of the first type.
It shows that $G$ is strongly $\Z_7$-connected, which
contradicts Lemma~\ref{CL: nostrongZ7}.
\end{proof}

\begin{claim}
\label{CL: Vge5}
$|G|\ge 5$.
\end{claim}
\begin{proof}
Suppose the claim is false.  Claim~\ref{CL: nostrongZ7}(b) implies $|G|=4$.
Since $\rho(G)\ge 0$, the trivial partition shows that $\|G\|\ge 19$.  First
suppose $\|G\|>19$, and let $G'=G-e$, for some arbitrary edge $e$.  Since
$\|G'\|<\|G\|$, we will apply
Lemma~\ref{LEM: smallH2}(c-i) to prove $G'$ is strongly $\Z_7$-connected.  Since $|G'|=4$, we know $G'\notin \F$. So it
suffices to show that $G'$
is 6-edge-connected and $\rho_{G'}(\P)\ge 3$ for every nontrivial partition
$\P$.  The first condition holds because $G$ is 8-edge-connected, by
Claim~\ref{CL: ess11}(a).  The second holds because $\rho_{G'}(\P)\ge
\rho_G(\P)-2\ge 5$, by Claim~\ref{CL: nontrivialge716}(a).  So $G'$ is strongly
$\Z_7$-connected by Lemma~\ref{LEM: smallH2}(c-i), which contradicts
Claim~\ref{CL: nostrongZ7}.

Instead assume $\|G\|=19$.  Claim~\ref{CL: ess11}(a) implies $\delta(G)\ge 8$.  Since
$G$ contains no $6K_2$ by Claim~\ref{CL: nostrongZ7}(a), we know $\mu(G)\le 5$.
Now Lemma~\ref{K4weightsZ7-lem} shows that $G$ is strongly $\Z_7$-connected.
Thus, $G$ is not a counterexample, which proves the claim.
\end{proof}

\begin{claim}\label{CL: noY2}
$G$ has no copy of $T_{1,1,5}\dit$.
\end{claim}
\begin{proof}
Suppose $G$ contains a copy of $T_{1,1,5}\dit$ with vertices $w,x,y,z$ and
$\mu(xy)=4$. We lift $xz, zy$ to become a new edge $xy$, and lift $xw,wy$ to become
another new edge $xy$, and then contract the resulting $6K_2$ to form a
new graph $G'$. The trivial partition ${\cal Q}^*$ of $G'$ satisfies
$\rho_{G'}({\cal Q}^*)\ge \rho(G)-2(8)+17\ge 1$. Every nontrivial
partition $\Q'$ of $G'$ corresponds to a normal partition $\Q$ of $G$
in which the contracted vertex is replaced by $\{x, y\}$. Since $xz, zy, xw, wy$
are the only edges possibly counted in $\rho_{G}({\cal Q})$ but not in
$\rho_{G'}({\Q'})$, Claim~\ref{CL: nontrivialge716}(b) implies
$\rho_{G'}({\Q'})\ge \rho_{G}({\Q})- 2(4)\ge 4$.  Since $w\neq z$,
Claim~\ref{CL: ess11}(a,b) implies $G'$ is $6$-edge-connected. Because
$|V(G')|=|V(G)|-1\ge 4$, we know $G'\not\cong T_{a,b,c}$ with
$a+b+c\in\{10,11\}$. Hence $G'$ is strongly $\Z_7$-connected by Lemma~\ref{LEM:
smallH2}(c-i).
This is a lifting reduction of the first type.
So $G$ is strongly $\Z_7$-connected, which is a contradiction.
\end{proof}

\begin{claim}
\label{CL: deltage10}
$G$ has minimum degree at least $10$. So $G$ is $10$-edge-connected by Claim~\ref{CL: ess11}.
\end{claim}
\begin{proof}
The second statement follows from the first.  To prove the first,
suppose there exists $x\in V(G)$ with $8\le d(x)\le 9$.
Let $x_1, x_2$ be two neighbors of $x$.  To form a graph $G'$ from $G$, we
lift $x_1x, xx_2$ to become a new edge $x_1x_2$, orient the
remaining edges incident with $x$ to achieve $\beta(x)$, and finally delete
$x$.  This is similar to achieving $\beta(v_1)$ in the proof of Lemma~\ref{LEM:
4K2J2K4inSZ5}  (that $G$ has no copy of $6K_2$).  This is a lifting reduction
of the second type.  So, to show $G$ has a $\beta$-orientation, it suffices
to show that $G'$ is strongly $\Z_7$-connected.

Observe that the trivial partition $\Q^*$ of $G'$ satisfies
$\rho_{G'}(\Q^*)\ge \rho(G)-2(9-1)+17\ge 1$.
Also, for an almost trivial partition $\Q'$ of $G'$ with $|Q_1|=2$, we have
$\rho_{G'}({\Q}')\ge \rho_{G'}({\Q}^*)+17-2(5)\ge 8$. Note that when
$Q_1=\{x_1,x_2\}$ we still have $\mu_{G'}(x_1x_2)\le 5$ by Claim~\ref{CL:
noY1}. Moreover, for any normal partition ${\Q}'$ of $G'$, since
$\Q=\Q'\cup\{x\}$ is a normal partition of $G$, we have $\rho_{G'}(\Q')\ge
\rho_G(\Q)-2(9)+17\ge 11$.  Since $|G'|=|G|-1\ge 4$ and $\rho_{G'}(\Q')\ge
8$ for any nontrivial partition, Lemma~\ref{LEM: smallH2}(a)
implies that $G'$ is strongly $\Z_7$-connected.
\end{proof}

\begin{claim}\label{CL: noT115dot}
$G$ has no copy of $T^{\,\bullet}_{1,1,5}$.
\end{claim}
\begin{proof}
Suppose $G$ has a copy of $T^{\,\bullet}_{1,1,5}$, with vertices
$v_1,v_2,v_3,v_4$ (in order around a 4-cycle) and $\mu(v_1v_4)=5$.
We lift the edges $v_1v_2, v_2v_3, v_3v_4$ to become a new copy of edge
$v_1v_4$ and contract the resulting $6K_2$; call this new graph $G'$.
The trivial partition ${\Q}^*$ of $G'$ satisfies $\rho_{G'}({\Q}^*)\ge
\rho(G)-2(8)+17\ge 1$. Every nontrivial partition ${\Q'}$ of $G'$
corresponds to a normal partition $\Q$ of $G$ in which the contracted vertex is
replaced by $\{v_1, v_4\}$.
Since $v_1v_2,v_2v_3,v_3v_4$ are the
only edges possibly counted in $\rho_G(\Q)$ but not in $\rho_{G'}(\Q')$, we have
$\rho_{G'}(\Q') \ge \rho_G(\Q)-2(3)\ge 6$ by Claim~\ref{CL: nontrivialge716}(b).
Claim~\ref{CL: Vge5} implies $|G'|=|G|-1\ge 4$, so $G'\notin \F$.  Since $G$ is
$10$-edge-connected by Claim \ref{CL: deltage10}, the graph $G'$ is
$6$-edge-connected.  So $G'$ is strongly $\Z_7$-connected by Lemma~\ref{LEM:
smallH2}(c-i).
\end{proof}

\begin{claim}
\label{CL: noY3}
$G$ has no copy of $T_{2,2,4}$.
\end{claim}
\begin{proof}
Suppose $G$ contains a copy of $T_{2,2,4}$ with vertices $x,y,z$ and $\mu(xy)=4$.
To form a new graph $G'$ from $G$, we delete two copies (each) of $xz, zy$ and
add two new parallel edges $xy$, and then contract the resulting $6K_2$ induced
by $\{x,y\}$.  Claim~\ref{CL: deltage10} shows $G'$
is $6$-edge-connected.  Similar to the proof of Claim~\ref{CL: noY2}, the
trivial partition ${\cal Q}^*$ of $G'$ satisfies $\rho_{G'}({\Q}^*)\ge
\rho(G)-2(8)+17\ge 1$, and every nontrivial partition $\Q'$ of $G'$ satisfies
$\rho_{G'}(\Q')\ge \rho_G(\Q)- 2(4)\ge 4$.  Since $|G'|=|G|-1\ge 4$,
Lemma~\ref{LEM: smallH2}(c-i) implies $G'$ is strongly $\Z_7$-connected.
This is a lifting reduction of the first type, which implies that
$G$ is strongly $\Z_7$-connected, and thus gives a contradiction.
\end{proof}

\begin{claim}
\label{CL: noalmostge14}
\label{CL: noalmostge15}
For any normal partition ${\P}=\{P_1, P_2,\dots, P_t\}$ with $|P_1|\ge 3$, we have
$$\rho_G({\P})\ge 14.$$
\end{claim}
\begin{proof}
Suppose the claim is false and let $\P$ be such a partition with
$\rho_G({\P})\le 13$.  Let $H=G[P_1]$. Since $G$ contains no
copy of $T_{1,1,5}$ or $T_{2,2,4}$, 
we know $H\not\cong T_{a,b,c}$ with $a+b+c\in\{10,11\}$ (and $\min\{a,b,c\}\ge
1$).  Thus, since $|H|=|P_1|\ge 3$,
we know $H\notin \F$.

Let ${\Q}=\{Q_1, Q_2,\dots, Q_s\}$ be a partition of $H$.
Now $\Q\cup (\P\setminus\{P_1\})$ is a partition of $G$, and 
Eq.~(\ref{EQ: wH2}) 
implies $\rho_H(\Q)= \rho_G(\Q\cup (\P\setminus\{P_1\})) - \rho_G(\P)+14\ge\rho_G({\Q}\cup
({\P}\setminus\{P_1\}))+1$.  If $\Q$ is a nontrivial partition of $H$, then
$\Q\cup (\P\setminus\{P_1\})$ is a nontrivial partition of $G$, and so
Claim~\ref{CL: nontrivialge716}(a) implies $\rho_H({\Q})\ge \rho_G({\Q}\cup
({\P}\setminus\{P_1\}))+1\ge 8$. If $\Q$ is the trivial partition of $H$,
then $\rho_H({\Q})\ge \rho_G({\Q}\cup ({\P}\setminus\{P_1\}))+1\ge 1$.
By Lemma~\ref{LEM: smallH2}(a), the subgraph $H$ is strongly $\Z_7$-connected,
which contradicts Claim~\ref{CL: nostrongZ7}.
\end{proof}

Now we can strengthen Claim~\ref{CL: ess11}(b).

\begin{claim}\label{CL: ess12}
If $[X, X^c]$ is an edge cut with $|X|\ge 2$ and $|X^c|\ge 3$, then $|[X,X^c]|\ge 12$.
\end{claim}
\begin{proof}
Let $X$ satisfy the hypotheses and let ${\P}=\{X, X^c\}$.  We will
prove $\rho_G(\P)\ge 21$.  Assume, to the contrary, that $\rho_G(\P)\le 20$.
Let $H=G[X]$ and let $\Q=\{Q_1,\ldots,Q_s\}$ be a partition of $H$.  Let
$\P'=\Q\cup\{X^c\}$.
Eq.~\eqref{EQ: wH2} implies $\rho_H(\Q)=\rho_G(\P')-\rho_G(\P)+14$.  Since
$|X^c|\ge 3$, Claim~\ref{CL: noalmostge14} implies $\rho_G(\P')\ge 14$.  Thus
$\rho_H(\Q)\ge 14-20+14=8$. By Lemma~\ref{LEM: smallH2}(b), subgraph $H$ is
strongly $\Z_7$-connected, which contradicts Claim~\ref{CL: nontrivialge716}(b).
So $21\le \rho_G({\cal P})=2|[X,X^c]|-34+31$, which implies $|[X,X^c]|\ge 12$.
\end{proof}

The value of Claim~\ref{CL: ess12} is that it allows us to lift three pairs of
edges (with at most two incident to a common vertex) and know that the
resulting graph $G'$ is still 6-edge-connected.  Thus, we will show that $G'$
is strongly $\Z_7$-connected, since it satisfies the hypotheses of
Lemma~\ref{LEM: smallH2}(c-i).

\begin{figure}[t]

\setlength{\unitlength}{0.08cm}

\begin{center}

\begin{picture}(155,40)

\put(10,0){\circle*{2}}\put(30,0){\circle*{2}}\put(10,20){\circle*{2}} \put(30,20){\circle*{2}} \put(20,32){\circle*{2}}\put(20,-10){\circle*{2}}

\qbezier(10, 0)(20, 4)(30, 0)\qbezier(10, 0)(20, -4)(30, 0)\qbezier(10, 0)(20, 1.5)(30, 0)\qbezier(10, 0)(20, -1.5)(30, 0)
\qbezier(10, 0)(8.5, 10)(10, 20)\qbezier(10, 0)(6, 10)(10, 20)\qbezier(10, 0)(14, 10)(10, 20)\qbezier(10, 0)(11.5, 10)(10, 20)

\qbezier(30, 20)(20, 16)(10, 20)\qbezier(30, 20)(20, 18.5)(10, 20)\qbezier(30, 20)(20, 21.5)(10, 20)\qbezier(30, 20)(20, 24)(10, 20)

\qbezier(30, 20)(26, 10)(30, 0)\qbezier(30, 20)(28.5, 10)(30, 0)\qbezier(30, 20)(34, 10)(30, 0)\qbezier(30, 20)(31.5, 10)(30, 0)
\qbezier(30, 20)(30, 20)(20, 32)\qbezier(10, 20)(10, 20)(20, 32)
\qbezier(30, 0)(30, 0)(20, -10)\qbezier(10, 0)(10, 0)(20, -10)

\put(60,0){\circle*{2}}\put(80,0){\circle*{2}}\put(60,20){\circle*{2}} \put(80,20){\circle*{2}} \put(70,32){\circle*{2}}

\qbezier(60, 0)(70, -1.5)(80, 0)\qbezier(60, 0)(70, 1.5)(80, 0)\qbezier(60, 0)(70, -4)(80, 0)\qbezier(60, 0)(70, 4)(80, 0)
\qbezier(60, 0)(58.5, 10)(60, 20)\qbezier(60, 0)(61.5, 10)(60, 20)\qbezier(60, 0)(56, 10)(60, 20)
\qbezier(60, 0)(64, 10)(60, 20)
\qbezier(80, 20)(70, 16)(60, 20)\qbezier(80, 20)(70, 24)(60, 20)\qbezier(80, 20)(70, 21.5)(60, 20)\qbezier(80, 20)(70, 18.5)(60, 20)

\qbezier(80, 20)(76, 10)(80, 0)\qbezier(80, 20)(78.5, 10)(80, 0)\qbezier(80, 20)(81.5, 10)(80, 0)\qbezier(80, 20)(84, 10)(80, 0)
\qbezier(80, 20)(80, 20)(70, 32)\qbezier(60, 20)(60, 20)(70, 32)
\qbezier(60, 0)(40, 20)(70, 32)\qbezier(80, 0)(100, 20)(70, 32)

\put(110,0){\circle*{2}}\put(140,0){\circle*{2}}\put(125,20){\circle*{2}} \put(143,20){\circle*{2}}\put(107,20){\circle*{2}}\put(125,-10){\circle*{2}}
\qbezier(140, 0)(125, -10)(125, -10)\qbezier(110, 0)(125, -10)(125, -10)
\qbezier(110, 0)(125, 4)(140, 0)\qbezier(110, 0)(125, -4)(140, 0)\qbezier(110, 0)(125, 0)(140, 0)
\qbezier(110, 0)(110, 0)(125, 20)\qbezier(110, 0)(115, 12)(125, 20)\qbezier(110, 0)(120, 8)(125, 20)

\qbezier(140, 0)(140, 0)(125, 20)\qbezier(140, 0)(130, 8)(125, 20)\qbezier(140, 0)(135, 12)(125, 20)
\qbezier(110, 0)(108.5, 10)(107, 20)\qbezier(125, 20)(116, 20)(107, 20)\qbezier(143, 20)(134, 20)(125, 20)\qbezier(140, 0)(141.5, 10)(143, 20)

\put(14,-20){\footnotesize{$(5C_4^=)\diit$}}\put(55.5,-20){\footnotesize{identified
$(5C_4^=)\diit$}}\put(121,-20){\footnotesize{$T_{4,4,4}\diiit$}}
\end{picture}
\end{center}
\vspace{1.4cm}
\caption{The graphs $(5C_4^=)\diit$, identified $(5C_4^=)\diit$, $T_{4,4,4}\diiit$.}
\label{FIG: YT444}
\end{figure}
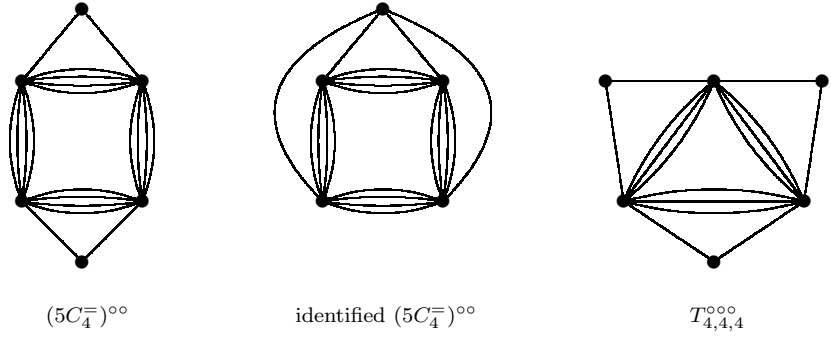

Recall that \Emph{$5C_4^=$} denotes the graph formed from $5C_4$ by removing
the edges of a perfect matching.

\begin{claim}
$G$ contains neither a copy of $(5C_4^=)\diit$ nor a copy of $(5C_4^=)\diit$ with its two 2-vertices
identified.
\label{CL: noYnoYid}
\end{claim}
\begin{proof}
Suppose $G$ contains a copy of $(5C_4^=)\diit$ with vertices $v_1,v_2,v_3,v_4,w_1,
w_2$, where $v_1,\ldots,v_4$ lie on the 4-cycle and $N(w_1)=\{v_1,v_2\}$ and
$N(w_2)=\{v_3,v_4\}$.  In $G$ we lift edges $v_1w_1,w_1v_2$ to form a new copy
of $v_1v_2$ and lift edges $v_3w_2,w_2v_4$ to form a new copy of $v_3v_4$;
call this new graph $G'$.
In $G'$ vertices $v_1,\ldots,v_4$ induce a copy of $5C_4^=$ (if either $v_1v_3$ or
$v_2v_4$ is present in $G$, then $G$ contains $T_{1,1,5}\dit$, which is a
contradiction).  Claim~\ref{K4weightsZ7-lem} implies $5C_4^=$ is strongly
$\Z_7$-connected.  Form $G''$ from $G'$ by contracting $\{v_1,v_2,v_3,v_4\}$.
Since $G$ is 10-edge-connected by Claim~\ref{CL: deltage10}, we know $G''$ is
6-edge-connected.  The trivial partition $\Q^*$ of $G''$
satisfies $\rho_{G''}(\Q^*)\ge \rho(G)+3(17)-2(20)\ge 11$.  Each nontrivial
partition $\Q''$ of $G''$ corresponds to a normal partition $\Q$ of $G$ in which
the contracted vertex is replaced by $\{v_1,v_2,v_3,v_4\}$.  Since at most four
edges are counted in $\rho_G(\Q)$ but not in $\rho_{G''}(\Q'')$, we have
$\rho_{G''}(\Q'')\ge \rho_G(\Q)-2(4)\ge 6$ by
Claim~\ref{CL: noalmostge14}.
Thus, $\rho(G'')\ge 6$, so Lemma~\ref{LEM: smallH2}(c-ii) implies that $G''$ is
strongly $\Z_7$-connected, and also that $G$ is strongly $\Z_7$-connected,
which is a contradiction.  If vertices $w_1$ and $w_2$ are identified, the same
proof works, since Claim~\ref{CL: deltage10} still implies that $G''$ is
6-edge-connected.
\end{proof}

\begin{claim}
\label{CL: noT444'''}
$G$ contains no copy of $T_{4,4,4}\diiit$.
\end{claim}
\begin{proof}
Suppose $G$ contains a copy of $T_{4,4,4}\diiit$ with vertices
$v_1,v_2,v_3,w_1,w_2,w_3$ and $d(v_i)=8$ and $d(w_i)=2$ for all $i$ and
$N(w_i)=\{v_1,v_2,v_3\}\setminus\{v_i\}$.  Form $G'$ from $G$ by lifting the
pair of edges incident to each vertex $w_i$ and contracting the resulting
$T_{4,4,4}$.  This is a lifting reduction of the first type.  Since $T_{4,4,4}$
is strongly $\Z_7$-connected by Lemma~\ref{Z7-contract-configs}, it suffices to
show that $G'$ is also strongly $\Z_7$-connected.  Claims~\ref{CL: ess12}
and~\ref{CL: deltage10} imply that $G'$ is 6-edge-connected.
The trivial partition $\P^*$ of $G'$
satisfies $\rho_{G'}(\P^*)\ge \rho(G)+17(2)-2(15)\ge 4$.  Each nontrivial
partition $\P'$ of $G'$ corresponds to a normal partition $\P$ of $G$ in which
the contracted vertex is replaced by $\{v_1,v_2,v_3\}$.  We show below that for
such a partition we can strengthen Claim~\ref{CL: noalmostge14} to
$\rho_G(\P)\ge 15$.  Then we have $\rho_{G'}(\P')\ge \rho_G(\P)-2(6)\ge 3$ by
Claim~\ref{CL: nontrivialge716}(b), since at most six edges are counted in $\rho_G(\P)$ but
not in $\rho_{G'}(\P')$.  Thus, $\rho(G')\ge 3$, so Lemma~\ref{LEM:
smallH2}(c-ii) implies that $G'$ is strongly $\Z_7$-connected, which is a
contradiction.  Now it suffices to show that $\rho_G(\P)\ge 15$.

Suppose, to the contrary, that $\rho_G(\P)\le 14$.  Let $P_1$ be the part of
$\P$ containing $\{v_1,v_2,v_3\}$, and let $H=G[P_1]$.  We will show that
$H$ is strongly $\Z_7$-connected, which gives a contradiction.
Let $\Q=\{Q_1,\ldots,Q_s\}$ be a partition of $H$.  Let $\P''=\Q\cup(\P\setminus
\{P_1\})$.  Eq.~\eqref{EQ: wH2} implies
$\rho_H(\Q)=\rho_G(\P'')-\rho_G(\P)+14\ge \rho_G(\P'')\ge 0$.  Further, if $\Q$
is a nontrivial partition of $H$, then $\P''$ is a nontrivial partition of $G$,
so Claim~\ref{CL: nontrivialge716} implies $\rho_H(\Q)\ge \rho_G(\P'')\ge 7$.
Since $H$ contains $T_{3,3,3}$ by construction, and $G$ does not contain
$T_{2,2,4}$, we know that $H\notin\F$.  To apply Lemma~\ref{LEM: smallH2}(c-i),
we show that $H$ is 6-edge-connected.  Consider a bipartition $\Q=\{Q_1,Q_2\}$
of $H$.  Since $\Q$ is nontrivial, $7\le \rho_G(\P'')\le
\rho_H(\Q)=2|[Q_1,Q_2]_H|-2(17)+31$, which implies $|[Q_1,Q_2]_H|\ge 5$.
That is, $H$ is 5-edge-connected.  If $H$ is 6-edge-connected,  then
Lemma~\ref{LEM: smallH2}(c-i) implies that $H$ is strongly $\Z_7$-connected,
which is a contradiction.  So assume $H$ has a bipartition $\Q=\{Q_1,Q_2\}$ with
$|[Q_1,Q_2]_H|=5$. By symmetry, we assume $|Q_1|\ge |Q_2|$.
Since $H$ contains $T_{3,3,3}$ and $T_{3,3,3}$ is 6-edge-connected, we know
that $|Q_1|\ge3$.  
Now $\rho_G(\P'')=\rho_G(\P)+2(5)-17\le 14-7=7$.  Since $\P''$ is normal with
$|Q_1|\ge 3$, this contradicts Claim~\ref{CL: nontrivialge716}.
\end{proof}

\subsection{Discharging}
Fix a plane embedding of a planar graph $G$ such that $\rho(G)\ge 0$.
(We assume that all parallel edges between two vertices $v$ and $w$ are
embedded consecutively, in the cyclic orders, around both $v$ and $w$.)
If $G$ has a cut-vertex, then each block of is strongly $\Z_7$-connected by
minimality, so $G$ is strongly $\Z_7$-connected by Lemma~\ref{reduc-lem},
which is a contradiction. Hence $G$ is $2$-connected.
Since $\rho(G)\ge 0$, we have $2\|G\|-17|G|+31\ge 0$.  By Euler's Formula,
$|G|+|F(G)|-\|G\|=2$.  Now solving for $|G|$ and substituting into the
inequality gives:
$$
\sum_{f\in F(G)}\ell(f)=2\|G\|\le \frac{34}{15}|F(G)|-\frac25.
$$
We assign to each face $f$ initial charge $\ell(f)$.  So the total charge is
strictly less than $34|F(G)|/15$.  To reach a contradiction, we redistribute
charge so that each face ends with charge at least $34/15$.
We use the following three discharging rules.

\begin{enumerate}
\item[(R1)] Each 2-face takes charge $2/15$ from each weakly adjacent
$3^+$-face.
\item[(R2)] Each 3-face takes charge $2/15$ from each weakly adjacent
$4^+$-face with which its parallel edge has multiplicity at most 3 and $1/15$
from each weakly adjacent $4^+$-face with which its parallel edge has
multiplicity 4.
\item[(R3)] After (R1) and (R2), each 3-face with more than $34/15$ splits its
excess equally among weakly adjacent 3-faces with less than $34/15$.
\end{enumerate}

Now we show that each face ends with charge at least $34/15$.  By (R1) each
2-face ends with $2+2(2/15)=34/15$.  Consider a $5^+$-face $f$.
Since $G$ contains no copy of $6K_2$, each edge of $f$ has mutliplicity at most
5.  Since $G$ contains no copy of $T_{1,1,5}$, face $f$ sends at most $4(2/15)$
across each of its edges.  Thus $f$ ends with at least
$\ell(f)-4(2/15)\ell(f)=7\ell(f)/15\ge 35/15$.  Consider a 4-face $f$.  Since
$G$ contains no copy of $T^{\,\bullet}_{1,1,5}$, each edge of $f$
has multiplicity at most 4.  So $f$ sends at most $3(2/15)+1/15=7/15$ across
each of its edges.  If $f$ sends at most 5/15 across one edge, then $f$ ends
with at least $4-3(7/15)-5/15=34/15$.  If $f$ sends at most $6/15$ across at
least two of its edges, then $f$ ends with at least $4-2(7/15)-2(6/15)=34/15$.
So assume that neither of these cases holds.  Thus, each edge of $f$ has
multiplicity 4, and $f$ is weakly adjacent to 3-faces across at least three of
its edges.  This contradicts Claim~\ref{CL: noYnoYid}.

Let $f$ be a 3-face $T_{a,b,c}$.  If $a+b+c\le 8$, then $f$ ends (R2) with at least
$3-(8-3)(2/15)=35/15$.  So assume $a+b+c\ge 9$.  Since $G$ has no $T_{1,1,5}$,
we know $\max\{a,b,c\}\le 4$.  Since $G$ has no $T_{2,2,4}$, if $\max\{a,b,c\}=4$,
then $\min\{a,b,c\}=1$.  Thus, each 3-face $T_{a,b,c}$ finishes (R1) with excess
charge at least $1/15$ unless $T_{a,b,c}\in\{T_{1,4,4},T_{3,3,3}\}$.
So we only need to consider $T_{1,4,4}$ and $T_{3,3,3}$.
Suppose $f$ is $T_{1,4,4}$.  Each face adjacent to $f$ across an edge of
multiplicity 4 is not a 3-face, since $G$ has no $T_{1,1,5}\dit$.  So $f$ ends (R2)
with at least $3-(9-3)(2/15)+2(1/15)=35/15$.
Hence, each 3-face $f$ ends (R2) with at least $35/15$ unless $f$ is $T_{3,3,3}$.

So assume that $f$ is $T_{3,3,3}$.  If any adjacent face is not a
3-face, then $f$ ends (R2) with at least $3-(9-3)(2/15)+2/15=35/15$.  So assume
each adjacent face is a 3-face.  If these three adjacent faces do not intersect
outside $f$, then $G$ contains a copy of $T_{4,4,4}\diiit$, a contradiction.
If all three faces intersect outside $f$, then $|V(G)|=4$, which contradicts
Claim~\ref{CL: Vge5}.  So assume that exactly two faces
adjacent to $f$ intersect outside $f$.  Let $f_1$ and $f_2$ denote the
3-faces adjacent to $f$ that intersect outside $f$.  Denote
the boundaries of $f$, $f_1$, and $f_2$ by (respectively) $vwx$, $vwy$,
and $wxy$.
Suppose $\mu(wy)\ne 3$.  Now $f_1$ and $f_2$ each end (R2) with at least
$35/13$, so by (R3) each gives $f$ at least $(1/2)(1/15)$.  Thus $f$ ends happy.
So assume $\mu(wy)=3$.  Now $d(w)=3+3+3$, which contradicts that $\delta(G)\ge
10$, by Claim~\ref{CL: deltage10}.  This completes the proof.

\newpage
\section*{Appendix: Proofs of Lemmas \ref{Z7-contract-configs} and \ref{K4weightsZ7-lem}}
\noindent{\bf Lemma \ref{Z7-contract-configs}.} {\em Each of the following graphs is strongly $\Z_7$-connected: $6K_2$,
$3K_4^{+}$, and every 6-edge-connected graph $T_{a,b,c}$ where $a+b+c=12$.}

\begin{proof}
Throughout we fix a $\Z_7$-boundary $\beta$ and construct an orientation to
achieve $\beta$.

Let $G=6K_2$, with $V(G)=\{v_1,v_2\}$. To achieve
$\beta(v_1)\in\{0,1,2,3,4,5,6\}$, the number of edges we orient out of $v_1$ is
(respectively) 3, 0, 4, 1, 5, 2, 6.

Let $G=T_{a,b,c}$, with $a+b+c=12$ and $\delta(G)\ge 6$.  (We handle this before
$3K_4^+$.)  Let
$V(G)=\{v_1,v_2,v_3\}$.  If $G$ contains a 6-vertex, say  $v_1$, then $\mu(v_2v_3)=6$. Since $G/v_2v_3\cong6K_2$ is strongly $\Z_7$-connected, $G$ is strongly $\Z_7$-connected by Lemma \ref{reduc-lem}(ii). So assume that $\delta(G)\ge 7$.  If $G$ contains a 7-vertex
$v_i$ and $\beta(v_i)\ne 0$, then we orient $5$ edges incident to $v_i$ to achieve
$\beta(v_i)$, and lift the remaining pair of nonparallel edges to form a new
edge.  We are done, since $6K_2$ is strongly $\Z_7$-connected.
If $G$ contains an 8-vertex $v_j$ and $\beta(v_j)\notin\{1,6\}$, then we orient
4 edges incident to $v_j$ to achieve $\beta(v_j)$, and lift two pairs of
nonparallel edges to form new edges.  Again we are done, since $6K_2$ is
strongly $\Z_7$-connected.  Since $\|G\|=12$ and $\delta(G)\ge 7$, the possible
  degree sequences of $G$ are (a) $\{7,7,10\}$, (b) $\{7,8,9\}$, and (c)
$\{8,8,8\}$.  The edge multiplicities of $G$ are the three values $\|G\|-d(v_i)$.
So $G$ is (a) $T_{2,5,5}$, (b) $T_{3,4,5}$, or (c) $T_{4,4,4}$.
In each case we assume $d(v_1)\le d(v_2)\le d(v_3)$.  In (a) we may
assume $\beta(v_1)=\beta(v_2)=0$, which implies $\beta(v_3)=0$.
To achieve this boundary, orient all edges out of $v_1$ and all edges into $v_2$.
In (b) we may assume $\beta(v_1)=0$, $\beta(v_2)=1$, and $\beta(v_3)=6$.
To achieve this boundary, orient all edges out of $v_2$ and all edge into $v_1$.
(If instead $\beta(v_2)=6$ and $\beta(v_3)=1$, then we reverse the direction of all edges.)
In (c) we assume $\beta(v_i)\in\{1,6\}$ for all $i$.  This
yields a contradiction, since $\sum_{i=1}^3\beta(v_i)\equiv 0\pmod{7}$.

Let $G=3K_4^+$, with $V(G)=\{v_1,v_2,v_3,v_4\}$ and $d(v_1)=d(v_2)=9$
and $d(v_3)=d(v_4)=10$.  Similar to the previous paragraph, we may assume
$\beta(v_1)=\beta(v_2)=0$, $\beta(v_3)=1$, and $\beta(v_4)=6$.  (If not, then
we can lift some edges pairs at $v_i$ and use the remaining edges incident to
$v_i$ to achieve $\beta(v_i)$.) To achieve this boundary, start by orienting
all edges out of $v_1$, all edges into $v_2$, and all edges $v_4v_3$ out of
$v_4$.  Now reverse one copy of $v_3v_2$ and reverse one copy of $v_1v_4$.
\end{proof}

\noindent{\bf Lemma \ref{K4weightsZ7-lem}.}
{\em The graph $5C_4^=$ is strongly $\Z_7$-connected.  Further, if $G$ is a graph
with $|G|=4$, $\|G\|=19$, $\mu(G)\le 5$, and $\delta(G)\ge 8$, then $G$ is
strongly $\Z_7$-connected.}

\begin{proof}
Assume $G$ satisfies the hypotheses (either the first or second), and let $V(G)=\{v_1,v_2,v_3,v_4\}$.
Our plan is to form a new graph $G_i$ from $G$ by lifting one,
two, or three pairs of edges incident to $v_i$,  using the remaining edges
incident to $v_i$ to achieve the desired boundary $\beta(v_i)$ at $v_i$.
This is a lifting reduction of the second type.  If $\|G_i\|\ge 12$ and $G_i$
is 6-edge-connected, then $G_i$ is strongly $\Z_7$-connected by
Lemma~\ref{Z7-contract-configs}, and so we can find an orientation to achieve
the $\beta$ boundary of $G$.  We will show that in every case we can construct
such a $G_i$, and achieve $\beta(v_i)$ using edges incident to $v_i$ that are
not lifted to form $G_i$.

Denote $V(5C_4^=)$ by
$\{v_1,v_2,v_3,v_4\}$, with $N(v_1)=N(v_3)=\{v_2,v_4\}$, and fix a $\Z_7$-boundary
$\beta$.  If $\beta(v_1)\in \{1,3,4,6\}$, then we lift three pairs of edges
incident to $v_1$ and use the remaining edges to achieve $\beta(v_1)$. Notice that the resulting graph $G_1$ satisfies $\|G_1\|= 12$, and we are done in this case.
So, by symmetry, we assume $\beta(v_i)\in \{0,2,5\}$ for each $i$.  The possible
multisets of $\beta$ values are $\{0,0,0,0\}$, $\{0,0,2,5\}$, and $\{2,5,2,5\}$.
Up to symmetry, we have five possible $\Z_7$-boundaries.  Figure~\ref{pics2} shows
orientations that achieve these.

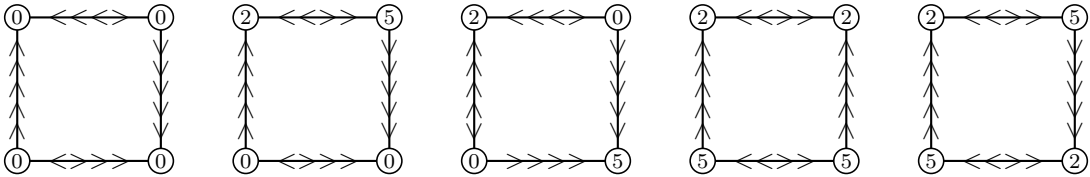
\begin{figure}[hb]
\centering

\tikzstyle{uStyle}=[shape = circle, minimum size = 6pt, inner sep = 1pt,
outer sep = 0pt, fill=white, semithick, draw]
\tikzstyle{lStyle}=[shape = circle, minimum size = 5pt, inner sep =
0.5pt, outer sep = 0pt, font=\footnotesize]
\tikzstyle{vlStyle}=[shape = circle, minimum size = 4pt, inner sep =
1.0pt, outer sep = 0pt, draw, fill=white, semithick, font=\footnotesize]
\def\spacer{5cm}

\begin{tikzpicture}[thick,scale=.75]
\def\rad{1.8cm}
\def\spacer{4.05cm}

\foreach \ang/\name/\boundary in {45/1/0, 135/2/0, 225/3/0, 315/4/0}
  \draw (\ang:\rad) node[vlStyle] (v\name) {\footnotesize{\boundary}};
\foreach \i/\j/\label in {1/2/{<<<>}, 2/3/{<<<<<}, 3/4/{<>>>}, 4/1/{<<<<<}}
  \draw (v\i) -- (v\j) node[midway, sloped]{$\label$};

\begin{scope}[xshift=\spacer]
\foreach \ang/\name/\boundary in {45/1/5, 135/2/2, 225/3/0, 315/4/0}
  \draw (\ang:\rad) node[vlStyle] (v\name) {\footnotesize{\boundary}};
\foreach \i/\j/\label in {1/2/{<<>>}, 2/3/{<<<<<}, 3/4/{<>>>}, 4/1/{<<<<<}}
  \draw (v\i) -- (v\j) node[midway, sloped]{$\label$};
\end{scope}

\begin{scope}[xshift=2*\spacer]
\foreach \ang/\name/\boundary in {45/1/0, 135/2/2, 225/3/0, 315/4/5}
  \draw (\ang:\rad) node[vlStyle] (v\name) {\footnotesize{\boundary}};
\foreach \i/\j/\label in {1/2/{<<<>}, 2/3/{<<<<>}, 3/4/{>>>>}, 4/1/{<<<<<}}
  \draw (v\i) -- (v\j) node[midway, sloped]{$\label$};
\end{scope}

\begin{scope}[xshift=3*\spacer]
\foreach \ang/\name/\boundary in {45/1/2, 135/2/2, 225/3/5, 315/4/5}
  \draw (\ang:\rad) node[vlStyle] (v\name) {\footnotesize{\boundary}};
\foreach \i/\j/\label in {1/2/{<<>>}, 2/3/{<<<<<}, 3/4/{<<>>}, 4/1/{>>>>>}}
  \draw (v\i) -- (v\j) node[midway, sloped]{$\label$};
\end{scope}

\begin{scope}[xshift=4*\spacer]
\foreach \ang/\name/\boundary in {45/1/5, 135/2/2, 225/3/5, 315/4/2}
  \draw (\ang:\rad) node[vlStyle] (v\name) {\footnotesize{\boundary}};
\foreach \i/\j/\label in {1/2/{<<>>}, 2/3/{<<<<<}, 3/4/{<<>>}, 4/1/{<<<<<}}
  \draw (v\i) -- (v\j) node[midway, sloped]{$\label$};
\end{scope}

\end{tikzpicture}
\caption{Orientations achieving the possible boundaries with
$\beta(v_i)\in\{0,2,5\}$ for all $i$.\label{pics2}}
\end{figure}

Now we prove the second statement.
Suppose $G$ contains an 8-vertex $v_i$.  To form $G_i$, we lift one (arbitrary,
nonparallel) pair of edges incident to $v_i$.  Now $\|G_i\|= 19-8+1=12$.  If
$G_i$ contains a copy of $6K_2$, then we are done by Lemma~\ref{reduc-lem}, since
$6K_2$ is strongly $\Z_7$-connected, and contracting this copy of $6K_2$ yields
another $6K_2$.  So instead we assume $\mu(G_i)\le 5$.
The edge-connectivity of $G_i$ is $\delta(G_i)=\|G_i\|-\mu(G_i)\ge 12-5=7$.
Since $G_i$ is 6-edge-connected, we are done by Lemma~\ref{Z7-contract-configs}.
Hence, we assume that $\delta(G)\ge 9$ below.

Suppose some pair $v_i,v_j$ of vertices has no edges joining it; that is,
$\mu(v_iv_j)=0$.  By symmetry, we assume $i=1$ and $j=2$.  Since
$d(v_1)\ge 9$ and $d(v_2)\ge 9$, we get that $\mu(v_1v_3)+\mu(v_1v_4)\ge 9$ and
$\mu(v_2v_3)+\mu(v_2v_4)\ge 9$.  Since $G$ has no $6K_2$, each edge of the
4-cycle $v_1v_3v_2v_4$ has multiplicity at least 4.  Either $\mu(v_1v_3)=5$ or
$\mu(v_1v_4)=5$; by symmetry we assume the latter.  If $\mu(v_3v_4)=1$, then
we lift edge $v_1v_3,v_3v_4$ to form a new copy of $v_1v_4$.  We contract the
resulting $6K_2$ induced by $\{v_1,v_4\}$.  The resulting graph $G'$ is
$T_{3,4,5}$, so we are done by Lemmas~\ref{reduc-lem} and \ref{Z7-contract-configs}.  Instead assume
$\mu(v_1v_3)=0$.  Now $G=5C_4^-$ (formed from $5C_4$ by deleting a single
edge). Thus $G$ contains $5C_4^=$ as a spanning subgraph, and so $G$ is  strongly $\Z_7$-connected by
Lemma~\ref{Z7-contract-configs}.
Thus, we assume $\mu(v_iv_j)\ge 1$ for all distinct $i,j\in[4]$.

Suppose $\mu(v_iv_j)=5$ for some distinct $i,j\in[4]$; by symmetry, say $\mu(v_1v_1)=5$.  Since $\mu(v_1v_3)\ge 1$ and $\mu(v_2v_3)\ge 1$, we lift one copy of
each of $v_1v_3$ and $v_3v_2$ to form a new copy of $v_1v_2$, and then contract
$\{v_1,v_2\}$ (calling the new vertex $w$).  Denote this new graph by $G'$.  We
show that $G'$ is strongly $\Z_7$-connected, which implies the result for $G$
by Lemma~\ref{reduc-lem}, since $6K_2$ is strongly $\Z_7$-connected.
We first show that $G$ is $8$-edge-connected.  Each edge
cut separating a single vertex $v_i$ has size $d(v_i)\ge \delta(G)\ge 8$.
If an edge cut $S$ separates $G$ into two parts of size 2, then $|S|\ge
\|G\|-2\mu(G) \ge 19-2(5)=9$.  Thus, $G$ is 8-edge-connected, which implies that
$G'$ is 6-edge-connected.  Since $\|G\|=19$, we have $\|G'\|=19-7=12$.
So $G'$ is strongly $\Z_7$-connected, by Lemma~\ref{Z7-contract-configs}.
Thus $G$ is strongly $\Z_7$-connected by Lemma~\ref{reduc-lem}(ii).
This implies that $\mu(v_iv_j)\le 4$ for each pair $i,j\in[4]$.

Suppose that $\mu(v_iv_j)=1$ for some pair $i,j\in[4]$;  say
$\mu(v_1v_2)=1$.  Since $d(v_1)\ge 9$ and $d(v_2)\ge 9$ and $\mu(G)\le 4$, we
have $\mu(v_1v_3)=\mu(v_1v_4)=\mu(v_2v_3)=\mu(v_2v_4)=4$.  Since $\|G\|=19$,
this implies $\mu(v_3v_4)=2$; see Case 1 in Figure~\ref{figK4abcdef}.  By
orienting 5 edges incident to a vertex $v_i$ we can achieve any boundary
value $\beta(v_i)$ other than 0.  So if
$\beta(v_1)\ne 0$ or $\beta(v_2)\ne 0$, then we achieve it by orienting 5
incident edges, and lifting two pairs of incident edges to reduce to a
6-edge-connected subgraph $G_i$ with $\|G_i\|=12$.  Similarly, by orienting
4 edges incident to a vertex $v_i$ we can achieve any boundary value at
$v_i$ other than 1 or 6.  So if $\beta(v_3)\notin\{1,6\}$ or
$\beta(v_4)\notin\{1,6\}$, then we achieve $\beta(v_i)$ by orienting 4 edges
incident to $v_i$ and lifting 3 pairs of incident edges; we do this so that the
three newly created edges in $G_i$ are not all parallel.  Since $\mu(G)\le 4$ we
have $\mu(G_i)\le 6$.  Now we can finish on $G_i$, by
Lemma~\ref{Z7-contract-configs}.  Thus, by symmetry between $v_3$ and $v_4$, we
assume $\beta(v_1)=\beta(v_2)=0$, $\beta(v_3)=1$, and $\beta(v_4)=6$.  Case~1
in Figure~\ref{figK4abcdef} shows an orientation achieving this boundary.  So in what remains
we assume that $\mu(v_iv_j)\ge 2$ for each pair $i,j\in[4]$.

Since $\|G\|=19$ and $\delta(G)\ge 9$, the degree sequence is either
$\{9,9,9,11\}$ or $\{9,9,10,10\}$.  Suppose we are in the first case.  By
symmetry, we assume $d(v_4)=11$, $\mu(v_1v_4)=\mu(v_2v_4)=4$, and $\mu(v_3v_4)=3$.
Since $d(v_1)=d(v_2)=d(v_3)=9$ and
$\mu(v_1v_2)+\mu(v_1v_3)+\mu(v_2v_3)=8$, we have $\mu(v_1v_2)=2$ and
$\mu(v_1v_3)=\mu(v_2v_3)=3$.  See Case~2 of Figure~\ref{figK4abcdef}.
If $\beta(v_i)\ne 0$ for any $i\in \{1,2,3\}$,
then we achieve $\beta(v_i)$ by orienting 5 edges incident to $v_i$, and we
lift two pairs of incident edges to form $G_i$, which is 6-edge-connected and has
$\|G_i\|=12$.  So we assume $\beta(v_1)=\beta(v_2)=\beta(v_3)=0$.  This implies
that also $\beta(v_4)=0$.  Case~2 in Figure~\ref{figK4abcdef} shows an orientation achieving this
boundary.

\begin{figure}[t]
\centering

\tikzstyle{uStyle}=[shape = circle, minimum size = 6pt, inner sep = 1pt,
outer sep = 0pt, fill=white, semithick, draw]
\tikzstyle{lStyle}=[shape = circle, minimum size = 5pt, inner sep =
0.5pt, outer sep = 0pt, font=\footnotesize]
\tikzstyle{vlStyle}=[shape = circle, minimum size = 4pt, inner sep =
1.0pt, outer sep = 0pt, draw, fill=white, semithick, font=\footnotesize]
\def\spacer{5cm}

\begin{tikzpicture}[thick,scale=.75]
\def\rad{2cm}

\draw (0,0) node[vlStyle] (v2) {\footnotesize{0}};
\foreach \ang/\name/\boundary in {90/1/0, 210/3/1, 330/4/6}
  \draw (\ang:\rad) node[vlStyle] (v\name) {\footnotesize{\boundary}};
\foreach \i/\j/\label in {1/2/{<}, 1/3/{<<<<}, 1/4/{>>>>}, 2/3/{<<<<}, 2/4/{<>>>},
3/4/{>>}}
  \draw (v\i) -- (v\j) node[midway, sloped]{$\label$};
\draw (0,-2) node {\footnotesize{Case 1}};

\begin{scope}[xshift=\spacer]
\draw (0,0) node[vlStyle] (v2) {\footnotesize{0}};
\foreach \ang/\name/\boundary in {90/1/0, 210/3/0, 330/4/0}
  \draw (\ang:\rad) node[vlStyle] (v\name) {\footnotesize{\boundary}};
\foreach \i/\j/\label in {1/4/{>>>>}, 1/2/{<>}, 1/3/{<<<}, 4/2/{>>>>}, 4/3/{<<>},
2/3/{<<<}}
  \draw (v\i) -- (v\j) node[midway, sloped]{$\label$};
\draw (0,-2) node {\footnotesize{Case 2}};
\end{scope}

\begin{scope}[xshift=2*\spacer]
\draw (0,0) node[vlStyle] (v2) {\footnotesize{0}};
\foreach \ang/\name/\boundary in {90/1/1, 210/3/0, 330/4/6}
  \draw (\ang:\rad) node[vlStyle] (v\name) {\footnotesize{\boundary}};
\foreach \i/\j/\label in {1/2/{<>}, 1/3/{<<<<}, 1/4/{>>>>}, 2/3/{<<<}, 2/4/{>>>>},
3/4/{<>}}
  \draw (v\i) -- (v\j) node[midway, sloped]{$\label$};
\draw (0,-2) node {\footnotesize{Case 3}};
\end{scope}

\begin{scope}[xshift=3*\spacer]
\draw (0,0) node[vlStyle] (v2) {\footnotesize{0}};
\foreach \ang/\name/\boundary in {90/1/0, 210/3/1, 330/4/6}
  \draw (\ang:\rad) node[vlStyle] (v\name) {\footnotesize{\boundary}};
\foreach \i/\j/\label in {1/2/{<>}, 1/3/{<<<}, 1/4/{>>>>}, 2/3/{<<<<}, 2/4/{>>>},
3/4/{<>>}}
  \draw (v\i) -- (v\j) node[midway, sloped]{$\label$};
\draw (0,-2) node {\footnotesize{Case 4}};
\end{scope}

\end{tikzpicture}
\caption{In each case $v_1$ is at top, $v_2$ center, $v_3$ left, and $v_4$
right.\label{figK4abcdef}}
\end{figure}
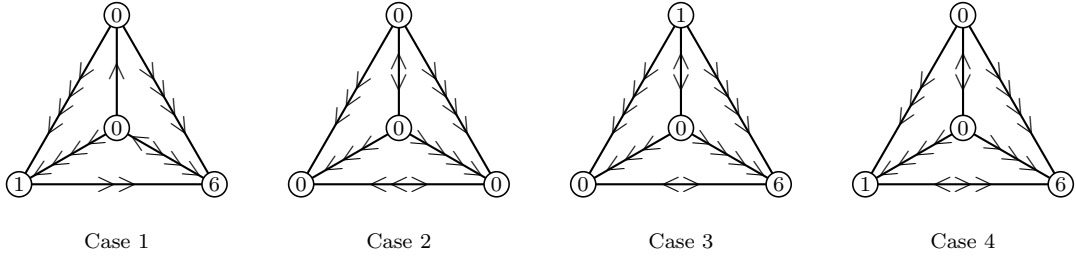

Finally, assume the degree sequence is $\{9,9,10,10\}$ and $\mu(v_iv_j)\ge
2$ for each pair $i,j\in[4]$.  If $\mu(v_iv_j)\ge 3$ for each pair
$i,j\in[4]$ then $G\cong 3K_4^+$, which contradicts
Lemma~\ref{Z7-contract-configs}.  So
assume by symmetry that $\mu(v_1v_2)=2$.  First suppose that $d(v_1)=10$.
This implies $\mu(v_1v_3)=\mu(v_1v_4)=4$.  Since each edge has multiplicity
2, 3, or 4, we cannot have $d(v_2)=10$ (because otherwise $\mu(v_3v_4)=1$).  So $d(v_2)=9$ and, by symmetry between
$v_3$ and $v_4$, we assume $d(v_3)=9$ and $d(v_4)=10$.  This implies that
$\mu(v_2v_3)=3$, $\mu(v_2v_4)=4$, and $\mu(v_3v_4)=3$; see Case 3 of Figure~\ref{figK4abcdef}.
As above, we can lift
two or three pairs of incident edges if either $\beta(v_2)\ne 0$, $\beta(v_3)\ne 0$,
$\beta(v_1)\notin\{1,6\}$, or $\beta(v_4)\notin\{1,6\}$.  So we assume
$\beta(v_2)=\beta(v_3)=0$, $\beta(v_1)=1$, and $\beta(v_4)=6$.  (If, instead,
$\beta(v_2)=\beta(v_3)=0$, $\beta(v_1)=6$, and $\beta(v_4)=1$, then we can
achieve this by reversing every edge.)  The desired orientation is shown in
Case~3 of Figure~\ref{figK4abcdef}.

Again assume the  degree sequence is $\{9,9,10,10\}$ and that
$\mu(v_1v_2)=2$.  Rather than as above, we now assume $d(v_1)=d(v_2)=9$.  So
$d(v_3)=d(v_4)=10$.  By symmetry between $v_3$ and $v_4$ (and also between $v_1$
and $v_2$) we assume $\mu(v_1v_3)=\mu(v_2v_4)=3$, $\mu(v_1v_4)=\mu(v_2v_3)=4$,
and $\mu(v_3v_4)=3$.  For the same reasons as in the previous paragraph, we
assume $\beta(v_1)=\beta(v_2)=0$, $\beta(v_3)=1$, and $\beta(v_4)=6$.  Now the
desired orientation is shown in Case~4 of Figure~\ref{figK4abcdef}.  This completes the proof.
\end{proof}


{\footnotesize
\begin{thebibliography}{s2}

\bibitem{BHIKW2008}
O.V. Borodin, S. G. Hartke, A. O. Ivanova, A. V. Kostochka, and D. B. West,
Circular $(5,2)$-coloring of sparse graphs, {\em Sib. Elektron. Mat. Izv.}, 5 (2008),
 417-426.


\bibitem{BIK2007}
O.V. Borodin, A.O. Ivanova and A.V. Kostochka, Oriented $5$-coloring of sparse plane graphs, {\em J.
Applied and Industrial Math.} 1(1) (2007), 9-17.

\bibitem{BKKW2004}
O.V. Borodin, S.-J. Kim, A.V. Kostochka and D.B. West,
Homomorphisms from sparse graphs with large girth, \JCTB 90 (2004) p.147-159.

\bibitem{BKNRS1999}
O. V. Borodin, A. V. Kostochka, J. Ne\v{s}et\v{r}il, A. Raspaud, and E. Sopena,
On the maximum average degree and the oriented chromatic number of
a graph, \DM 206 (1999), 77-90.

\bibitem{DP2017}
Z. Dvo\v{r}\'{a}k and L. Postle, Density of $5/2$-critical graphs, \ComHung  37 (2017)
863-886.

\bibitem{EJLT2017}
L. Esperet, R. D. Joannis De Verclos, T. Le and S. Thomass\'{e}, Additive bases and flow on graphs, \SIAMDM 32(1) (2018), 534-542.

\bibitem{Goddyn1998}
  L. A. Goddyn, M. Tarsi and C.-Q. Zhang, On $(k,d)$-Colorings and
Fractional Nowhere-Zero Flows \JGT
28 (1998) 155-161.

\bibitem{HLWZ}
M. Han, J. Li, Y. Wu and C.-Q. Zhang, Counterexamples to Jaeger's Circular Flow
Conjecture, \JCTB 131 (2018). 1-11.

\bibitem{Jaeger1988} F. Jaeger, Nowhere-zero flow problems, in {\em Selected Topics in Graph Theory 3}, (L. W. Beineke and R. J. Wilson eds.), Academic Press, London, (1988) 71-95.

\bibitem{Lai2007}
H.-J. Lai, Mod $(2p+1)$-orientations and $K_{1, 2p+1}$-decompositions, \SIAMDM
21 (2007), 844--850.

\bibitem{Lai2014}
H.-J. Lai, Y. Liang, J. Liu, J. Meng, Y. Shao and Z. Zhang, On Strongly
$\Z_{2s+1}$-connected Graphs, \DAM 174 (2014) 73--80.

\bibitem{LaLL17}
J. Li, H.-J. Lai and R. Luo, Group Connectivity, Strongly $\Z_m$-Connectivity and Edge-Disjoint Spanning Trees, \SIAMDM 31 (2017) 1909--1922.


\bibitem{LTWZ2013} L.M. Lov\'asz, C. Thomassen, Y. Wu, and C.-Q. Zhang, Nowhere-zero $3$-flows and modulo $k$-orientations,
\JCTB 103 (2013) 587--598.

\bibitem{Nash1961} C. St. J. A. Nash-Williams, Edge-disjoint spanning trees of finite graphs,
\JLMS 36 (1961) 445--450.

\bibitem{NRS1997}
 J. Ne\v{s}et\v{r}il, A. Raspaud, and E. Sopena, Colorings and girth of oriented
planar graphs, \DM 165-166 (1997), 519--530.

\bibitem{NR1999}
J. Ne\v{s}et\v{r}il  and A. Raspaud, Antisymmetric flows and strong colourings of oriented graphs, {\em Ann. Inst.
Fourier} 49(3) (1999), 1037--1056.

\bibitem{Thomassen2012} C. Thomassen, The weak $3$-flow conjecture and the weak circular flow conjecture,
\JCTB   102 (2012)
521--529.

\bibitem{Tutte1961} W. T. Tutte, On the problem of decomposing a graph into $n$
connected factors, \JLMS 36 (1961) 221--230.

\bibitem{Zhang-book} C.-Q. Zhang, {\it Integer Flows and Cycle Covers of
Graphs}, Marcel Dekker, New York, 1997.

\bibitem{CQsplitting02} C.-Q. Zhang, Circular flows of nearly eulerian graphs
and vertex-splitting, \JGT  40 (2002), 147--161.

  \end{thebibliography}
}
\end{document}